\RequirePackage{fix-cm}
\documentclass[smallextended]{svjour3}       
\smartqed  
\usepackage{graphicx}
\usepackage{amsmath}
\usepackage{amssymb}
 \usepackage{mathptmx}      
\usepackage{latexsym}

\newtheorem{Theorem}{Theorem}[section]

\newtheorem{Lemma}{Lemma}[section]

\newtheorem{Assumption-Notation}[Theorem]{Assumption-Notation}

\newtheorem{Proposition}{Proposition}[section]
\newtheorem{Remark}{Remark}[section]
\newtheorem{Corollary}{Corollary}[section]
\newtheorem{Example}{Example}[section]

\numberwithin{Remark}{section}
\numberwithin{Lemma}{section}
\numberwithin{Theorem}{section}
\numberwithin{Proposition}{section}
\numberwithin{Corollary}{section}
\numberwithin{equation}{section}

\begin{document}

\title{Occupation times of general L\'evy processes}

\author{Lan Wu \and Jiang Zhou \and Shuang Yu}
\institute{Lan Wu \at
              School of Mathematical Sciences, Peking University, Beijing 100871, P.R.China \\
              \email{lwu@pku.edu.cn}                
           \and
           Jiang Zhou \at
             School of Mathematical Sciences, Peking University, Beijing 100871, P.R.China \\
              Tel.: +86 18810515809 \\
              \email{1101110056@pku.edu.cn}
              \and
              Shuang Yu \at
              School of Mathematical Sciences, Peking University, Beijing 100871, P.R.China \\
              \email{1101110054@pku.edu.cn}
}

\date{Received: date / Accepted: date}

\maketitle

\begin{abstract}
For an arbitrary L\'evy process $X$ which is not a compound Poisson process,
we are interested in its occupation times. We use a quite novel and useful
approach to derive  formulas for the Laplace transform of the joint distribution
of $X$ and its occupation times. Our formulas are compact, and more importantly,
the forms of the formulas clearly demonstrate the essential quantities for the
calculation of occupation times of $X$. It is believed that our results are
important not only for the study of stochastic processes, but also for financial applications.
\subclass{MSC 60}
\end{abstract}

\section{Introduction}
\label{sec:1}
The investigation of occupation times of stochastic processes is an interesting
and historic question. In $1939$, Paul L\'evy derived an interesting and useful
result:
\begin{equation}
\mathbb P\left(\int_{0}^{t}\textbf{1}_{\{W_u \geq 0 \}}du \in ds\right)
=\frac{ds}{\pi\sqrt{s(t-s)}},
\ \ 0 < s < t,
\end{equation}
where $(W_t)_{t\geq 0}$ is a standard Brownian motion and $\textbf{1}_A$ is the
indicator function of a set $A$; see L\'evy [13] for the details. After that, the
investigation on occupation times of L\'evy processes (in particular spectrally
negative L\'evy processes) has made much great progress. For example, the Laplace
transform of $\int_{0}^{\infty}\textbf{1}_{\{X_t < 0\}}dt$ and the joint Laplace
transform of $\tau_0^-$ and $\int_{0}^{\tau_0^-}\textbf{1}_{\{a < X_t < b \}}dt$
have been derived, where $X=(X_t)_{t\geq 0}$ is a spectrally negative L\'evy process;
$\tau_0^-$ is the first passage time of $X$ and $0\leq a \leq b$.  The interested
readers are referred to [12,15] for more details.

There are many papers considering the joint distribution of a L\'evy process and
its occupation times. For instance, for  a spectrally negative L\'evy process $X$,
the Laplace transform of
$\mathbb E_x\left[e^{-p\int_{0}^{t}\textbf{1}_{\{a < X_s < b \}}ds}
\textbf{1}_{\{X_t \in dy\}}\right]$ with respect to $t$ has been considered in [11].
Recently, Wu and Zhou [18] studied a similar problem, where the process $X$
is assumed to be a hyper-exponential jump diffusion process. Here, we want to mention
that Cai et al. [5] have derived formulas for
\[
\int_{0}^{\infty}e^{-(a+r)t}\mathbb E_x\left[e^{-\rho \int_{0}^{t}\textbf{1}_{\{X_s \leq b\}}ds
+\gamma X_t}\right]dt,\]
where $X$ is a double exponential jump diffusion process.

The above mentioned papers can be classified into two categories according to
the assumption on the process $X$. Some papers assume that the process $X$ is a
spectrally negative L\'evy process (e.g., [11,15]), the others allow the process $X$
to have two-sided jumps but pose a limitation on its jumps (the jumps of $X$ follow
exponential or hyper-exponential distributions). These two categories both have some
drawbacks. For the first category, the results in those papers are written in terms
of q-scale functions, which are associated to spectrally negative L\'evy processes;
thus it is very difficult to extend their results and approaches to the case that the
process $X$ has both positive and negative jumps. For the second one, the derivation
in these papers are heavily dependent on the assumption of exponential-type jump
distributions; therefore, it is likely that their approaches cannot be used to other
non-exponential-type jump distributions.

In this paper, for an arbitrary L\'evy process $X$ but not a compound Poisson
process, we explore the problem how to compute the following quantity:
\begin{equation}
\mathbb E_x\left[e^{-p \int_{0}^{t}\textbf{1}_{\{ X_s \leq  b \}}ds}\textbf{1}_{
\{X_t \in dy\}}\right],
\end{equation}
where $p$ is an appropriate constant. Formulas for the Laplace transform of (1.2)
with respect to $t$ are derived by applying a novel but straightforward approach.
Our method consists of two steps. First, we consider the case that $X$ is a jump
diffusion process with jumps having rational Laplace transform. Then  the result
is extended to a general L\'evy process via an approximation discussion. As in
[11,18], the result in this article has some financial applications. Specifically,
our results can be used in pricing occupation time derivatives. It is expected to obtain
some unusual and profound outcomes on the pricing of occupation time derivatives through
the application of the general result obtained in this paper. But here, we do not intend
to discuss this application further and leave it to future research.

The remainder of the paper is organized as follows. Some important preliminary
results related to L\'evy processes  are given in Section 2, and the main results
 are presented in Section 3. In the next two sections, details on the derivation are
 presented. Finally, we present two examples in Section 6 and draw some conclusions in Section 7.

\section{Some preliminary results}
\label{sec:2}
In this paper, we let $X=(X_t)_{t\geq 0}$ represent a general L\'evy process.
The law and the corresponding expectation of $X$ such that $X_0=x$ are denoted
respectively by $\mathbb P_x$ and $\mathbb E_x$. To simplify the notation, we
write $\mathbb P$ and $\mathbb E$ when $x = 0$. In addition, define
$\underline{X}_{T}:=\inf_{0\leq t \leq T}X_t$ and
$\overline{X}_{T}:=\sup_{0\leq t \leq T}X_t$
for $T \geq 0$, and  denote
\begin{equation}
\int_{a}^{b}:=\int_{(a,b)},  \ \ \int_{a^-}^{b}:=\int_{[a,b)} \ \
and \ \ \int_{a}^{b^+}:=\int_{(a,b]},
\end{equation}
where $a,b \in \mathbb R \bigcup \{-\infty,\infty\}$.

Throughout this article, we assume that $X$ is not a compound Poisson process;
 and the random variable  $e(q)$ for $q>0$, independent of $X$, is an exponential
 random variable with rate $q$; $Re(x)$ and $Im(x)$ represent the real part and
 the imaginary part of a complex number $x$, respectively.

The following lemma, which is taken from Proposition $15$ on page $30$ in
Bertion [3], is important for the derivation in the paper.

\begin{Lemma}
For any $q>0$ and $z \in \mathbb R$, we have $\mathbb P\left(X_{e(q)}=z\right)=0$,
 which leads to that $\mathbb P\left(X_t=z\right)=0$ for Lebesgue almost every $t>0$.
\end{Lemma}

\begin{Remark}
If $X$ is a compound Poisson process, then it is possible that
$\mathbb P\left(X_{e(q)}=z \right) > 0$ for some $z \in \mathbb R$, which will make
the discussion  more difficult.
\end{Remark}

The result in  Lemma 2.2 is well-known, one can see, e.g., Theorem 5 on page 160 in
Bertion [3].

\begin{Lemma}
For $Re(\xi) \leq 0$ and $q > 0$, we have
\begin{equation}
\begin{split}
\mathbb E\left[e^{\xi \overline {X}_{e(q)}}\right]
&=e^{\int_{0}^{\infty}\frac{1}{t}e^{-qt}\int_{0}^{\infty}(e^{\xi x}-1)
\mathbb P\left(X_t\in dx\right)dt}.
\end{split}
\end{equation}
\end{Lemma}

The following theorem gives some simple but useful results, and its derivation
is straightforward.

\begin{Theorem}
(i) For $p,q>0$, there exists an infinitely divisible distribution
$G_1(x)$ on $[0, \infty)$ with the  Laplace transform
\begin{equation}
\int_{0^-}^{\infty} e^{-s x}dG_1(x)=\frac{\mathbb E\left[e^{-s \overline{X}_{e(q)}}\right]}
{\mathbb E\left[e^{-s \overline{X}_{e(p+q)}}\right]}
=e^{\int_{0}^{\infty}(e^{-s x}-1) \Pi_1(dx)}, \ \ s\geq 0,
\end{equation}
where $\Pi_1(dx)$ is the L\'evy measure and  is given by
\begin{equation}
\Pi_1(dx)=\int_{0}^{\infty}\frac{1}{t}e^{-qt}\left(1-e^{-p t}\right)
\mathbb P\left(X_t \in dx\right)dt, \ \ x>0.
\end{equation}

(ii) For $q > 0$ and $-q< p < 0$, there are two measures $G_{21}(x)$ and
$G_{22}(x)$ on $[0, \infty)$ such that
\begin{equation}
\int_{0^-}^{\infty}e^{-s x}dG_{21}(x)=\frac{1}{2}\left(e^{-\int_{0}^{\infty}
e^{-sx} \Pi_2(dx)}+e^{\int_{0}^{\infty}e^{-sx} \Pi_2(dx)}\right), \ \ s > 0,
\end{equation}
and
\begin{equation}
\int_{0^-}^{\infty}e^{-s x}dG_{22}(x)=\frac{1}{2}\left(e^{\int_{0}^{\infty}
e^{-sx} \Pi_2(dx)}-e^{-\int_{0}^{\infty}e^{-sx} \Pi_2(dx)}\right), \ \ s > 0,
\end{equation}
where
\begin{equation}
\Pi_2(dx)=\int_{0}^{\infty}\frac{1}{t}e^{-qt}\left(e^{-p t}-1\right)
\mathbb P\left(X_t \in dx\right)dt, \ \ x>0.
\end{equation}

Besides, it holds that
\begin{equation}
\begin{split}
&e^{\int_{0}^{\infty}\Pi_2(dx)}\int_{0^-}^{\infty}e^{-s x}d(G_{21}(x)-G_{22}(x))\\
&=\frac{\mathbb E\left[e^{-s \overline{X}_{e(q)}}\right]}
{\mathbb E\left[e^{-s \overline{X}_{e(p+q)}}\right]}=
e^{\int_{0}^{\infty}(1-e^{-s x}) \Pi_2(dx)}, \ \ s > 0,
\end{split}
\end{equation}
where
\begin{equation}
\begin{split}
&\int_{0}^{\infty}\Pi_2(dx)=\int_{0}^{\infty}\frac{1}{t}e^{-qt}
\left(e^{-pt}-1\right)\mathbb P\left(X_t >0\right)dt< \infty.
\end{split}
\end{equation}

(iii) $G_1(x)$, $G_{21}(x)$ and $G_{22}(x)$  are continuous on $(0,\infty)$.
\end{Theorem}

\begin{proof}
(i)  According to (2.2), we can derive
\begin{equation}
\begin{split}
\frac{\mathbb E\left[e^{-s \overline{X}_{e(q)}}\right]}
{\mathbb E\left[e^{-s \overline{X}_{e(p+q)}}\right]}=e^{\int_{0}^{\infty}(e^{-s x}-1)\Pi_1(dx)},
\ \ for \ \ s \geq 0,
\end{split}
\end{equation}
where $\Pi_1(dx)$ is given by (2.4) and is a measure (since $p>0$). Note that
\begin{equation}
\Pi_1(0,\infty):=\int_{0}^{\infty}\Pi_1(dx)=\int_{0}^{\infty}\frac{1}{t}
e^{-qt}\left(1-e^{-p t}\right)\mathbb P\left(X_t > 0\right)dt < \infty.
\end{equation}
Therefore,  from the L\'evy-Khintchine formula (see, e.g., Theorem 8.1
on page 37 in Sato [16]), we obtain that the right-hand side of (2.10) is the Laplace
transform of an infinitely divisible distribution, i.e., there is an infinitely divisible
distribution $G_1(x)$ on $[0, \infty)$ such that
\begin{equation}
\int_{0^-}^{\infty} e^{-s x}dG_1(x)=e^{\int_{0}^{\infty}(e^{-s x}-1) \Pi_1(dx)}, \ \ s\geq 0.
\end{equation}
Formula (2.3) is derived from (2.10) and (2.12).

(ii) It follows from (2.2) that
\begin{equation}
\begin{split}
\frac{\mathbb E\left[e^{-s \overline{X}_{e(q)}}\right]}
{\mathbb E\left[e^{-s \overline{X}_{e(p+q)}}\right]}=
e^{\int_{0}^{\infty}(1-e^{-s x})\Pi_2(dx)}, \ \ s\geq 0,
\end{split}
\end{equation}
where $\Pi_2(dx)$ is given by (2.7) and is a measure (since $p<0$). Next, it is obvious that
\begin{equation}
\begin{split}
e^{-\int_{0}^{\infty}e^{-sx}\Pi_2(dx)}=
\sum_{n=0}^{\infty}
\frac{(-1)^{n}\int_{0}^{\infty}e^{-sx}d\Pi_2^{*n}(0,x)}{n!},
\end{split}
\end{equation}
where $d\Pi_2^{*0}(0,x)=\delta_0(dx)$; $\Pi_2(0,x):=\int_{0}^{x}\Pi_2(dy)$
and $\Pi_2^{*n}(0,x)$ for $n \geq 1$ is the n-fold convolution of $\Pi_2(0,x)$.

Therefore, for $x \geq 0$, we can define
\begin{equation}
G_{21}(x)=1+\sum_{n=1}^{\infty}\frac{\Pi_2^{*2n}(0,x)}{(2n)!} \ \
and \ \ G_{22}(x)=\sum_{n=1}^{\infty}\frac{\Pi_2^{*(2n-1)}(0,x)}{(2n-1)!},
\end{equation}
which are measures on $[0,\infty)$ obviously. Formulas (2.5) and (2.6) follow directly
from (2.15), and formula  (2.8) is due to (2.5), (2.6) and (2.13).

(iii) Formula (2.3) gives
\begin{equation}
\begin{split}
&\int_{0^-}^{\infty} e^{-s x}dG_1(x)=e^{\int_{0}^{\infty}(e^{-s x}-1) \Pi_1(dx)}\\
&=e^{-\Pi_1(0,\infty)}\sum_{n=0}^{\infty}
\frac{\int_{0}^{\infty}e^{-sx}d\Pi_1^{*n}(0,x)}{n!},
\end{split}
\end{equation}
where $d\Pi_1^{*0}(0, x)=\delta_0(dx)$; $\Pi_1(0,x):=\int_{0}^{x}\Pi_1(dy)$
and $\Pi_1^{*n}(0,x)$ for $n\geq 1$ is the n-fold convolution of $\Pi_1(0,x)$.
From (2.4) and Lemma 2.1, we know $\Pi_1(dx)$ has no atoms.Thus, $G_1(x)$ is continuous
on $(0,\infty)$.

Since $\Pi_2(dx)$ has no atoms, the conclusion that $G_{21}(x)$ and $G_{22}(x)$ are
continuous on $(0,\infty)$ can be seen from (2.15).
\qed
\end{proof}

\begin{Remark}
If we define $G_1(x)=0$ for $x < 0$, then from (2.3), we obtain that $G_1(x)$ is not
left-continuous at $0$ since
\begin{equation}
\begin{split}
&G_1(0)=\lim_{s\uparrow \infty}\int_{0^-}^{\infty} e^{-s x}dG_1(x)
=e^{-\int_{0}^{\infty}\frac{1}{t}e^{-qt}\left(1-e^{-p t}\right)\mathbb P\left(X_t > 0\right)dt}>0.
\end{split}
\end{equation}
\end{Remark}

\begin{Remark}
From Theorem 2.1 (ii),  it is easy to derive that
\begin{equation}
G_{21}(0)=1, \ \  G_{21}(\infty):=\lim_{x \uparrow \infty}G_{21}(x)=\frac{1}{2}
\left(e^{-\int_{0}^{\infty}\Pi_2(dx)}+e^{\int_{0}^{\infty}\Pi_2(dx)}\right),
\end{equation}
and
\begin{equation}
G_{22}(0)=0, \ \  G_{22}(\infty):=\lim_{x \uparrow \infty}G_{22}(x)=\frac{1}{2}
\left(e^{\int_{0}^{\infty}\Pi_2(dx)}-e^{-\int_{0}^{\infty}\Pi_2(dx)}\right).
\end{equation}
In particular,
\begin{equation}
G_{22}(\infty)<G_{21}(\infty)<e^{\int_{0}^{\infty}\Pi_2(dx)} \leq e^{\int_{0}^{\infty}
\frac{1}{t}e^{-qt}(e^{-pt}-1)dt}.
\end{equation}
\end{Remark}

Corollary 2.1 states a similar result to Theorem 2.1, and can be proved by applying Theorem 2.1
to the dual process $-X$.

\begin{Corollary}
(i) For $p,q> 0$, there exists an infinitely divisible distribution $G_3(x)$ on $[0, \infty)$,
whose  Laplace transform is given by
\begin{equation}
\int_{0^-}^{\infty} e^{-s x}dG_3(x)=\frac{\mathbb E\left[e^{s \underline{X}_{e(q)}}\right]}
{\mathbb E\left[e^{s \underline{X}_{e(p+q)}}\right]}
=e^{\int_{0}^{\infty}(e^{-s x}-1) \Pi_3(dx)}, \ \ s\geq 0,
\end{equation}
where
\begin{equation}
\Pi_3(dx)=\int_{0}^{\infty}\frac{1}{t}e^{-qt}\left(1-e^{-pt}\right)
\mathbb P\left(-X_t \in dx\right)dt,\ \ x>0.
\end{equation}

(ii) For $q>0$ and $-q<p<0$, there are  two measures $G_{41}(x)$ and $G_{42}(x)$ on $[0, \infty)$
such that
\begin{equation}
\int_{0^-}^{\infty}e^{-s x}dG_{41}(x)=\frac{1}{2}\left(e^{-\int_{0}^{\infty}
e^{-sx} \Pi_4(dx)}+e^{\int_{0}^{\infty}e^{-sx} \Pi_4(dx)}\right), \ \ s > 0,
\end{equation}
and
\begin{equation}
\int_{0^-}^{\infty}e^{-s x}dG_{42}(x)=\frac{1}{2}\left(e^{\int_{0}^{\infty}
e^{-sx} \Pi_4(dx)}-e^{-\int_{0}^{\infty}e^{-sx} \Pi_4(dx)}\right), \ \ s > 0,
\end{equation}
where $\Pi_4(dx)$ is a measure and is given by
\begin{equation}
\begin{split}
&\Pi_4(dx)=\int_{0}^{\infty}\frac{1}{t}e^{-qt}\left(e^{-pt}-1\right)
\mathbb P\left(-X_t \in d x\right)dt,\ \ x>0.
\end{split}
\end{equation}

(iii) $G_3(x)$, $G_{41}(x)$ and $G_{42}(x)$ are continuous on $(0,\infty)$.
\end{Corollary}

\begin{proof}
For $t\geq 0$, let $X^1_t=-X_t$.  If $p,q>0$,  Theorem 2.1 (i) leads to that
 there is an infinitely divisible distribution $G_3(x)$ on $[0, \infty)$ such that
\begin{equation}
\int_{0^-}^{\infty} e^{-s x}dG_3(x)=\frac{\mathbb E\left[e^{-s \overline{X^1}_{e(q)}}\right]}
{\mathbb E\left[e^{-s \overline{X^1}_{e(p+q)}}\right]}
=e^{\int_{0}^{\infty}(e^{-s x}-1) \Pi_3(dx)}, \ \ s\geq 0,
\end{equation}
where
\begin{equation}
\Pi_3(dx)=\int_{0}^{\infty}\frac{1}{t}e^{-qt}\left(1-e^{-pt}\right)
\mathbb P\left(X^1_t \in d x\right)dt.
\end{equation}
Then, formulas (2.21) and (2.22) are followed after replacing $X^1$ by $-X$ in (2.26)
 and (2.27). The proofs of (ii) and (iii) are similar, thus we omit the details.
\qed
\end{proof}

\section{Main results}
\label{sec:3}
In this section, we first give a primary result (Theorem 3.1) in subsection 3.1, and
then present some corollaries in subsection 3.2.

\subsection{A primary result}
\label{sec:3.1}
For given $y \geq b$, $q>0$ and $p>-q$, define
\begin{equation}
V_q(x):=\mathbb E_x\left[e^{-p\int_{0}^{e(q)}\rm{\bf{1}}_{\{X_s \leq b\}}ds}
\rm{\bf{1}}_{\{X_{e(q)}>y\}}\right], \ \ x \in \mathbb R.
\end{equation}

The following Theorem 3.1 is the primary result of this paper, whose proof is very long
and is postponed  to the later two sections.

\begin{Theorem}
For  $q>0$ and $p>-q$, we have
\begin{equation}
V_q(x)-\mathbb P_x\left(X_{e(q)}> y\right)=J_1(b-x;y-b), \ \ y\geq b,
\end{equation}
with
\begin{equation}
J_1(x;y-b)=\int_{-\infty}^{x}F_1(x-z+y-b)dK_{q}(z), \ \ x \in \mathbb R,
\end{equation}
where $K_{q}(x)$ is the convolution of the probability distribution functions of
$\underline{X}_{e(q)}$ and $\overline{X}_{e(p+q)}$ under $\mathbb P$, i.e.,
\begin{equation}
K_{q}(x)=\int_{-\infty}^{\min\{0,x\}^+}\mathbb P\left(\overline{X}_{e(p+q)} \leq x-z\right)
\mathbb P\left(\underline{X}_{e(q)} \in dz \right), \ \ x \ \in \mathbb R,
\end{equation}
and $F_1(x)$ is a continuous function on $(0, \infty)$ with the Laplace transform
\begin{equation}
\begin{split}
&\int_{0}^{\infty} e^{-s x}F_1(x)dx
=\frac{1}{s}\left(\frac{\mathbb E\left[e^{-s \overline{X}_{e(q)}}\right]}
{\mathbb E\left[e^{-s \overline{X}_{e(p+q)}}\right]}-1\right), \ \ s > 0.
\end{split}
\end{equation}
\end{Theorem}

Before going further, we give some properties of $K_q(x)$ in (3.4) and $F_1(x)$ in (3.5)
in Propositions 3.1 and 3.2, respectively. These two propositions are important for the
rest of the paper.

\begin{Proposition}
$K_{q}(x)$ in (3.4) is continuous  on  $(-\infty, \infty)$.
\end{Proposition}

\begin{proof}
If $0$ is regular for $(0,\infty)$ or $(-\infty,0)$, i.e., $\mathbb P\left(\tau^+=0\right) =1$
or $\mathbb P\left(\tau^-=0\right) =1$, where $\tau^+=\inf\{t>0, X_t > 0\}$ and
$\tau^-=\inf\{t>0, X_t < 0\}$, then $\mathbb P\left(\overline{X}_{e(q)}=z\right)=0$ or
$\mathbb P\left(\underline{X}_{e(q)}=z\right)=0$ for any  $q>0$ and all $z \in \mathbb R$\footnote{For some $z \in \mathbb R$ and $T>0$, if $\int_{0}^{T}\rm{\bf{1}}_{\{\overline{X}_{t}=z\}}dt>0$, then there is at least one interval $(a,b)$ such that $\overline{X}_t=z$ for all $t \in (a,b)$ as the paths of $\overline{X}$ are non-decreasing. Since $0$ is regular for $(0,\infty)$, the probability $\mathbb P\left(\overline{X}_{t}=z \ \ for\ \  all \ \ t \in (a,b)\right)$ is zero, where $a<b$. This gives
$\mathbb E\left[\int_{0}^{T}\rm{\bf{1}}_{\{\overline{X}_{t}=z\}}dt\right]=0$ for all $T>0$ and $z \in \mathbb R$, thus $\mathbb P\left(\overline{X}_{e(q)}=z\right)=0$ for all $q>0$. },
thus the result that $K_{q}(x)$  is  continuous  on $\mathbb R$ is followed.

From Theorem 6.5 on page 142 and Corollary 6.6 on page 144 in Kyprianou [10], we
obtain that $0$ is irregular for both $(0,\infty)$ and $(-\infty,0)$ (i.e.,
$\mathbb P\left(\tau^+=0\right)=\mathbb P\left(\tau^-=0\right)=0$) only when $X$ is a compound
Poisson process. Since the compound Poisson process is excluded in this paper, it must hold
that $0$ is regular for $(0,\infty)$ or $(-\infty,0)$. Thus the desired result is derived.
\qed
\end{proof}

\begin{Proposition}
(i) For $q>0$ and $p > -q$, it holds that
\begin{equation}
F_1(x)=\left\{\begin{array}{cc}
G_1(x)-1, & if \ \ p > 0,\\
e^{\int_{0}^{\infty}\Pi_2(dx)}\Big(G_{21}(x)-G_{22}(x)\Big)-1, & if \ \ p < 0,
\end{array}
\right.
\end{equation}
where $G_1(x)$, $\Pi_2(dx)$, $G_{21}(x)$  and $G_{22}(x)$ are given by Theorem 2.1.
Moreover,
\begin{equation}
F_1(0):=\lim_{x\downarrow 0}F_1(x)=e^{-\int_{0}^{\infty}\frac{1}{t}e^{-qt}
\left(1-e^{-p t}\right)\mathbb P\left(X_t > 0\right)dt}-1,
\ \ \ F_1(\infty):=\lim_{x \uparrow \infty}F_1(x)=0.
\end{equation}

(ii) For $p > 0$, $F_1(x)$ is continuous, increasing and bounded on $[0,\infty]$, and
\begin{equation}
-1 \leq F_1(x)\leq 0, \ \ for \ \ any \ \ x \geq 0.
\end{equation}

(iii) For $-q< p < 0$ and $x\geq 0$,
\begin{equation}
|F_1(x)| < 2 e^{2\int_{0}^{\infty}\frac{1}{t}e^{-qt}(e^{-pt}-1)dt}+1.
\end{equation}
\end{Proposition}

\begin{proof}
(i) Applying integration by parts to (2.3) leads to
\begin{equation}
\int_{0}^{\infty} e^{-s x}G_1(x)dx=\frac{1}{s}\frac{\mathbb E\left[e^{-s \overline{X}_{e(q)}}\right]}
{\mathbb E\left[e^{-s \overline{X}_{e(p+q)}}\right]},
\end{equation}
which combined with (3.5), yields
\begin{equation}
\int_{0}^{\infty} e^{-s x}F_1(x)dx=\int_{0}^{\infty} e^{-s x}\left(G_1(x)-1\right)dx, \ \ s>0.
\end{equation}
This gives $F_1(x)=G_1(x)-1$, thus (3.6) holds for $p > 0$.  Similarly, from (2.8), we can
show that (3.6) is also valid  for $p < 0$.

Then, noting that
\begin{equation}
F_1(0)=\lim_{s \uparrow \infty}\int_{0}^{\infty}se^{-s x}F_1(x)dx \ \
and \ \ F_1(\infty)=\lim_{s \uparrow 0}\int_{0}^{\infty}s e^{-s x}F_1(x)dx,
\end{equation}
we can derive (3.7) from (2.3), (2.8) and (3.5).

(ii) This result can be obtained from (3.6) since $G_1(x)$ is a probability distribution
function and is continuous (see Theorem 2.1 (i)).

(iii) This result is due to (2.20) and (3.6) since
\[|F_1(x)| < e^{\int_{0}^{\infty}\Pi_2(dx)}2G_{21}(\infty)+1.\]
\qed
\end{proof}


\begin{Remark}
Since $G_1(x)$, $G_{21}(x)$  and $G_{22}(x)$ are measures. Formula (3.6) means that
$F_1(x)$ can be written as
\[
F_1(x)-F_1(0)=\int_{0}^{x}F_1(dz), \ \ x > 0,
\]
where  for $z > 0$,
\[
F_1(dz)=\left\{\begin{array}{cc}
G_1(dz), & if \ \ p > 0,\\
e^{\int_{0}^{\infty}\Pi_2(dx)}\Big(G_{21}(dz)-G_{22}(dz)\Big), & if \ \ p < 0.
\end{array}
\right.
\]
\end{Remark}

\begin{Remark}
Since $F_1(dx)$ for $x>0$ is well defined, for given $q>0$ and $p>-q$, formula (3.5) gives
\begin{equation}
\int_{0}^{\infty} e^{-s x}F_1(dx)+F_1(0)+1
=\frac{\mathbb E\left[e^{-s \overline{X}_{e(q)}}\right]}
{\mathbb E\left[e^{-s \overline{X}_{e(p+q)}}\right]},\ \  s>0.
\end{equation}
As $F_1(x)$ is bounded and continuous on $[0,\infty]$ (see Proposition 3.2), the above identity (3.13) can be extended to the half-plane $Re(s) \geq 0$. Particularly,
\begin{equation}
\int_{0^-}^{\infty} e^{i \phi x}d(F_1(x)+1)
=\frac{\mathbb E\left[e^{i \phi \overline{X}_{e(q)}}\right]}
{\mathbb E\left[e^{i \phi \overline{X}_{e(p+q)}}\right]},\ \ for \ \ \phi \in \mathbb R.
\end{equation}
\end{Remark}
\begin{Remark}
If $y=b$ in (3.1), then for fixed $b \in \mathbb R$, it follows from (3.2)--(3.5) and (3.14) that
\begin{equation}
\int_{-\infty}^{\infty}e^{-i \phi(x-b)}dV_q(x)=\mathbb E\left[e^{i \phi \underline{X}_{e(q)}}\right]\mathbb E\left[e^{i \phi
\overline{X}_{e(p+q)}}\right], \ \ \phi \in \mathbb R.
\end{equation}
In particular, if $p=0$, then (3.15) will reduce to the following well-known Wiener-Hopf factorization (see, e.g.,
Theorem 6.16 in [10])
\[
\mathbb E\left[e^{i \phi X_{e(q)}}\right]=\mathbb E\left[e^{i \phi \underline{X}_{e(q)}}\right]\mathbb E\left[e^{i \phi
\overline{X}_{e(q)}}\right], \ \ \phi \in \mathbb R.
\]
\end{Remark}

\subsection{Some corollaries}
\label{sec:3.2}
\begin{Corollary}
For $q>0$ and $p>-q$,
\begin{equation}
\mathbb E_{x}\left[e^{-p\int_{0}^{e(q)}\rm{\bf{1}}_{\{X_s \geq b\}}ds}
\rm{\bf{1}}_{\{X_{e(q)}< y\}}\right]-\mathbb P_x\left(X_{e(q)} < y\right)
=J_2(x-b;b-y),\ \ y \leq  b,
\end{equation}
with
\begin{equation}
J_2(x;b-y)=\int_{-\infty}^{x}F_2(x-z+b-y)d L_{q}(z),  \ \ x \in \mathbb R,
\end{equation}
where $L_{q}(x)$ is the convolution of the probability distribution functions of
$-\underline{X}_{e(p+q)}$ and $-\overline{X}_{e(q)}$ under $\mathbb P$, i.e.,
\begin{equation}
L_{q}(x)=\int_{\max\{0,x\}^-}^{\infty}\mathbb P\left(-\overline{X}_{e(q)} \leq x-z\right)
\mathbb P\left(-\underline{X}_{e(p+q)} \in dz \right), \ \ x \ \in \mathbb R,
\end{equation}
and $F_2(x)$ is a continuous function on $(0,\infty)$ with the Laplace transform
\begin{equation}
\begin{split}
&\int_{0}^{\infty} e^{-sz}F_2(z)dz
=\frac{1}{s}\left(\frac{\mathbb E\left[e^{s \underline{X}_{e(q)}}\right]}
{\mathbb E\left[e^{s \underline{X}_{e(p+q)}}\right]}-1\right), s > 0.
\end{split}
\end{equation}
\end{Corollary}

\begin{proof}
Consider the dual process $X^1_t=-X_t$ for $t\geq 0$. For $q>0$ and $p>-q$, the convolution
of the probability distribution functions of $\underline{X^1}_{e(q)}$ and
$\overline{X^1}_{e(p+q)}$ under $\mathbb P$ is given by $L_q(x)$ in (3.18).

As $y \leq b$, i.e., $-y \geq -b$, we can obtain from Theorem 3.1 that
\begin{equation}
\mathbb E_x\left[e^{-p\int_{0}^{e(q)}\rm{\bf{1}}_{\{X^1_s \leq -b\}}ds}
\rm{\bf{1}}_{\{X^1_{e(q)}>-y\}}\right]-\mathbb P_x\left(X^1_{e(q)}> -y\right)
=J_2(-b-x;b-y),
\end{equation}
with
\begin{equation}
J_2(x;b-y)=\int_{-\infty}^{x}F_2(x-z-y+b)d L_{q}(z), \ \ x \in \mathbb R,
\end{equation}
where $F_2(x)$ is given by (3.19) since
\[
\begin{split}
&\int_{0}^{\infty} e^{-s x}F_2(x)dx
=\frac{1}{s}\left(\frac{\mathbb E\left[e^{-s \overline{X^1}_{e(q)}}\right]}
{\mathbb E\left[e^{-s \overline{X^1}_{e(p+q)}}\right]}-1\right)=\frac{1}{s}
\left(\frac{\mathbb E\left[e^{s \underline{X}_{e(q)}}\right]}
{\mathbb E\left[e^{s \underline{X}_{e(p+q)}}\right]}-1\right).
\end{split}
\]
In addition, it is obvious that
\begin{equation}
\mathbb E_x\left[e^{-p\int_{0}^{e(q)}\rm{\bf{1}}_{\{X^1_s \leq -b\}}ds}
\rm{\bf{1}}_{\{X^1_{e(q)}>-y\}}\right]=\mathbb E_{-x}\left[e^{-p\int_{0}^{e(q)}
\rm{\bf{1}}_{\{X_s \geq b\}}ds}\rm{\bf{1}}_{\{X_{e(q)}< y\}}\right],
\end{equation}
and
\begin{equation}
\mathbb P_x\left(X^1_{e(q)}> -y\right)=\mathbb P_{-x}\left(X_{e(q)}< y\right).
\end{equation}

Thus, formula (3.16) is derived from (3.20) by first using the last two formulas and
then replacing $-x$ by $x$.
\qed
\end{proof}

\begin{Remark}
Similar to the derivation of (3.6), we can show that
\begin{equation}
F_2(x)=\left\{\begin{array}{cc}
G_3(x)-1, & if \ \ p > 0,\\
e^{\int_{0}^{\infty}\Pi_4(dx)}\Big(G_{41}(x)-G_{42}(x)\Big)-1, & if \ \ p < 0,
\end{array}
\right.
\end{equation}
where $G_3(x)$, $\Pi_4(dx)$, $G_{41}(x)$ and $G_{42}(x)$ are given by Corollary 2.1.
\end{Remark}

\begin{Remark}
Similar to the derivation of (3.15), for fixed $b \in \mathbb R$,  we can deduce the following result from (3.16).
\[
\int_{-\infty}^{\infty}e^{-i \phi(x-b)}d\left(\mathbb E_{x}\left[e^{-p\int_{0}^{e(q)}\rm{\bf{1}}_{\{X_s \geq b\}}ds}
\rm{\bf{1}}_{\{X_{e(q)}< b\}}\right]\right)=-\mathbb E\left[e^{i \phi \underline{X}_{e(p+q)}}\right]\mathbb E\left[e^{i \phi
\overline{X}_{e(q)}}\right], \ \ \phi \in \mathbb R.
\]
\end{Remark}

\begin{Corollary}
(i) For $p, q > 0$ and $y \geq b$, we have
\[
\begin{split}
&\mathbb E_x\left[e^{-p\int_{0}^{e(q)}\rm{\bf{1}}_{\{X_s > b\}}ds}
\rm{\bf{1}}_{\{X_{e(q)}>y\}}
\right]
=\frac{q}{p+q} \left(\mathbb P_x\left(X_{e(p+q)}> y\right)+\hat{J}_1(b-x;y-b)\right),
\end{split}
\]
where
\begin{equation}
\hat{J}_1(x;y-b)=\int_{-\infty}^{x}\hat{F}_1(x-z+y-b)d\hat{K}_{q}(z), \ \ x \in \mathbb R.
\end{equation}
Here, in (3.25), $\hat{K}_{q}(x)$ is the convolution of the probability distribution functions
of $\underline{X}_{e(p+q)}$ and $\overline{X}_{e(q)}$ under $\mathbb P$;  $\hat{F}_1(x)$  is
a continuous function on $(0,\infty)$ and satisfies
\begin{equation}
\begin{split}
&\int_{0}^{\infty} e^{-s x}\hat{F}_1(x)dx=\frac{1}{s}
\left(\frac{\mathbb E\left[e^{-s \overline{X}_{e(p+q)}}\right]}
{\mathbb E\left[e^{-s \overline{X}_{e(q)}}\right]}-1\right), \ \ s > 0.
\end{split}
\end{equation}

(ii) For $p,q > 0$ and $y \leq  b$,
\[
\begin{split}
&\mathbb E_{x}\left[e^{-p\int_{0}^{e(q)}\rm{\bf{1}}_{\{X_s < b\}}ds}
\rm{\bf{1}}_{\{X_{e(q)}< y\}}\right]=\frac{q}{p+q}\left(
\mathbb P_x\left(X_{e(p+q)} < y\right)+\hat{J}_2(x-b;b-y)\right),
\end{split}
\]
where
\begin{equation}
\hat{J}_2(x;b-y)=\int_{-\infty}^{x}\hat{F}_2(x-z-y+b)d \hat{L}_{q}(z), \ \ x \in \mathbb R.
\end{equation}
In (3.27), $\hat{L}_{q}(x)$ is the convolution of the probability distribution functions
of $-\underline{X}_{e(q)}$ and $-\overline{X}_{e(p+q)}$ under $\mathbb P$; $\hat{F}_2(x)$
is a continuous function on $(0,\infty)$ and satisfies
\begin{equation}
\begin{split}
\int_{0}^{\infty}e^{-s x}\hat{F}_2(x)dx
&=\frac{1}{s}\left(\frac{\mathbb E\left[e^{s \underline{X}_{e(p+q)}}\right]}
{\mathbb E\left[e^{s \underline{X}_{e(q)}}\right]}-1\right),\ \ s> 0.
\end{split}
\end{equation}
\end{Corollary}

\begin{proof}
Note first that
\begin{equation}
\mathbb E_x\left[e^{-p\int_{0}^{t}\rm{\bf{1}}_{\{X_s > b\}}ds}\rm{\bf{1}}_{\{X_{t}>y\}}\right]=
e^{-pt}\mathbb E_x\left[e^{-(-p)\int_{0}^{t}\rm{\bf{1}}_{\{X_s \leq  b\}}ds}
\rm{\bf{1}}_{\{X_{t}>y\}}\right].
\end{equation}
Then it holds that
\begin{equation}
\begin{split}
&\mathbb E_x\left[e^{-p\int_{0}^{e(q)}\rm{\bf{1}}_{\{X_{e(q)} > b\}}ds}
\rm{\bf{1}}_{\{X_{e(q)}>y\}}\right]\\
&=q\int_{0}^{\infty}e^{-qt}
e^{-pt}\mathbb E_x\left[e^{p\int_{0}^{t}\rm{\bf{1}}_{\{X_s \leq  b\}}ds}
\rm{\bf{1}}_{\{X_{t}>y\}}\right]dt\\
&=\frac{q}{p+q}\mathbb E_x\left[e^{-(-p)\int_{0}^{e(p+q)}\rm{\bf{1}}_{\{X_{s} \leq b\}}ds}
\rm{\bf{1}}_{\{X_{e(p+q)}>y\}}\right],
\end{split}
\end{equation}
which combined with (3.2) and (3.5), gives the results in the first part. The derivation of
the second part is similar and thus we omit the details.
\qed
\end{proof}

\begin{Remark}
Since Lemma 2.1 holds, the item $\int_{0}^{e(q)}\textbf{1}_{\{X_s < b\}}ds$ in (3.27) can be
rewritten as $\int_{0}^{e(q)}\textbf{1}_{\{X_s \leq b\}}ds$. A similar result holds for the
quantity $\int_{0}^{e(q)}\textbf{1}_{\{X_s > b\}}ds$ in (3.25).
\end{Remark}
\begin{Remark}
Although it is assumed that $p,q>0$ in Corollary 3.2, one can verify that this corollary also holds for $q>0$ and $p>-q$ by using a similar derivation in Theorem 3.1. The reason why we focus on the case  $p,q>0$ in Corollary 3.2 is that now it is a straightforward result of Theorem 3.1 and Corollary 3.1.
\end{Remark}

Similar to Remark 3.1, for the three functions $F_2(x)$ in (3.19), $\hat{F}_1(x)$ in (3.26)
and $\hat{F}_2(x)$ in (3.28), we have expressions for $F_2(dx)$, $\hat{F}_{1}(dx)$ and
$\hat{F}_2(dx)$ for $x > 0$. Particularly, Theorem 3.1 and Corollaries 3.1 and 3.2 will give
us the following result.

\begin{Corollary}
(i) For $p, q > 0$, we have
\begin{equation}
\begin{split}
&\mathbb E_x\left[e^{-p\int_{0}^{e(q)}\rm{\bf{1}}_{\{X_s \leq b\}}ds}
\rm{\bf{1}}_{\{X_{e(q)} \in dy\}}\right]\\
&=\left\{\begin{array}{cc}
\mathbb P_x\left(X_{e(q)}\in dy\right)- \int_{-\infty}^{b-x}F_1(dy-x-z)dK_{q}(z),& y \geq b,\\
\frac{q}{\xi}\left(\mathbb P_x\left(X_{e(\xi)} \in dy\right)-\int_{-\infty}^{x-b}
\hat{F}_2(x-dy-z)d \hat{L}_{q}(z)\right),& y \leq  b,
\end{array}\right.
\end{split}
\end{equation}
where $\xi=p+q$.

(ii) For $p, q > 0$, we have
\begin{equation}
\begin{split}
&\mathbb E_{x}\left[e^{-p\int_{0}^{e(q)}\rm{\bf{1}}_{\{X_s \geq b\}}ds}
\rm{\bf{1}}_{\{X_{e(q)}\in dy\}}\right]\\
&=\left\{\begin{array}{cc}
\mathbb P_x\left(X_{e(q)} \in dy\right)- \int_{-\infty}^{x-b}F_2(x-dy-z)d L_{q}(z), &y \leq  b,\\
\frac{q}{\xi}\left(\mathbb P_x\left(X_{e(\xi)} \in  dy\right)-\int_{-\infty}^{b-x}
\hat{F}_1(dy-x-z)d\hat{K}_{q}(z)\right), & y \geq b.
\end{array}\right.
\end{split}
\end{equation}
\end{Corollary}

\begin{Remark}
To obtain closed-form formulas for $\mathbb E_x\left[e^{-p\int_{0}^{e(q)}
\rm{\bf{1}}_{\{X_s \leq b\}}ds}\rm{\bf{1}}_{\{X_{e(q)} \in dy\}}\right]$ in (3.31), it is
enough to know the distributions of $\overline{X}_{e(q)}$ and $\underline{X}_{e(q)}$ for
any $q>0$, because these two distributions determine the distribution of $X_{e(q)}$
(which is due to the Wiener-Hopf factorization; see, e.g., Theorem 6.16 in [10]) and
other quantities, i.e., $F_1(x)$, $\hat{F}_2(x)$, $K_q(x)$ and $\hat{L}_q(x)$.
The distributions of $\overline{X}_{e(q)}$ and $\underline{X}_{e(q)}$ have been investigated
considerably, the reader can refer to [7,9].
\end{Remark}

\begin{Remark}
It follows from (3.14) that
\[
\int_{0}^{\infty} e^{i \phi x}F_1(dx)+F_1(0)+1
=\mathbb E\left[e^{i \phi \overline{X}_{e(q)}}\right]/
\mathbb E\left[e^{i \phi \overline{X}_{e(p+q)}}\right],\ \ for \ \ \phi \in \mathbb R.
\]
which combined with the definition of $K_q(x)$ in (3.4), gives
\[
\begin{split}
&\int_{-\infty}^{\infty}e^{i \phi x} \int_{-\infty}^{x}F_1(dx-z)dK_{q}(z)=\int_{0}^{\infty}e^{i \phi x}F_1(dx)\int_{-\infty}^{\infty}e^{i \phi x}dK_{q}(x)\\
&=\mathbb E\left[e^{i\phi X_{e(q)}}\right]-\left(F_1(0)+1\right)\int_{-\infty}^{\infty}e^{i \phi x}dK_{q}(x),\ \ \ \phi \in \mathbb R,
\end{split}
\]
where in the second equality, we have used the known Wiener-Hopf factorization (see Remark 3.3).
Therefore, for $x \in \mathbb R$, the last formula produces
\[
\int_{-\infty}^{x}F_1(dx-z)dK_{q}(z)=\mathbb P\left(X_{e(q)} \in dx\right)-(F_1(0)+1)K_q(dx).
\]

Similarly, we can derive that
\[
\begin{split}
&\int_{-\infty}^{x}\hat{F}_2(dx-z)d\hat{L}_{q}(z)=
\mathbb P\left(X_{e(p+q)} \in - dx\right)-(\hat{F}_2(0)+1)\hat{L}_q(dx),\\
&\int_{-\infty}^{x}\hat{F}_1(dx-z)d\hat{K}_{q}(z)=\mathbb P\left(X_{e(p+q)} \in dx\right)-(\hat{F}_1(0)+1)\hat{K}_q(dx),
\end{split}
\]
and
\[
\int_{-\infty}^{x}F_2(dx-z)dL_{q}(z)=\mathbb P\left(X_{e(q)} \in - dx\right)-(F_2(0)+1)L_q(dx).
\]
\end{Remark}

\begin{Remark}
From the above remark, we can write (3.31) and (3.32) as
\begin{equation}
\begin{split}
&\mathbb E_x\left[e^{-p\int_{0}^{e(q)}\rm{\bf{1}}_{\{X_s \leq b\}}ds}
\rm{\bf{1}}_{\{X_{e(q)} \in dy\}}\right]\\
&=\left\{\begin{array}{cc}
(F_1(0)+1)K_q(dy-x)+\int_{b-x}^{y-x}F_1(dy-x-z)dK_{q}(z),& y \geq b,\\
\frac{q}{p+q}\left((\hat{F}_2(0)+1)\hat{L}_q(x-dy)+\int_{x-b}^{x-y}
\hat{F}_2(x-dy-z)d \hat{L}_{q}(z)\right),& y \leq  b,
\end{array}\right.
\end{split}
\end{equation}
and
\begin{equation}
\begin{split}
&\mathbb E_{x}\left[e^{-p\int_{0}^{e(q)}\rm{\bf{1}}_{\{X_s \geq b\}}ds}
\rm{\bf{1}}_{\{X_{e(q)}\in dy\}}\right]\\
&=\left\{\begin{array}{cc}
(F_2(0)+1)L_q(x-dy)+\int_{x-b}^{x-y}F_2(x-dy-z)d L_{q}(z), &y \leq  b,\\
\frac{q}{p+q}\left((\hat{F}_1(0)+1)\hat{K}_q(dy-x)+\int_{b-x}^{y-x}
\hat{F}_1(dy-x-z)d\hat{K}_{q}(z)\right), & y \geq b,
\end{array}\right.
\end{split}
\end{equation}
where $q,p>0$.
\end{Remark}

\section{Proof of Theorem 3.1 in a dense subclass}
\label{sec:4}
In this section, the process $X=(X_t)_{t \geq 0}$ is assumed to be a L\'evy process with
Gaussian component and its jumps have rational Laplace transform, and one can refer to,
 e.g., Lewis and Mordecki [14] for the investigation on such processes. In specific,
  the process $X$ is given by
\begin{equation}
  X_t = X_0 + \mu t+\sigma W_t + \sum_{k=1}^{N^+_t}Z^+_k-\sum_{k=1}^{N^-_t}Z^-_k,
\end{equation}
where $\mu$ and $X_0$ are constants; $(W_t)_{t\geq 0}$ is a standard Brownian motion
with $W_0=0$, and $\sigma > 0$ is the volatility of the diffusion; $(N^+_t)_{t\geq 0}$ is
 a Poisson process with rate $\lambda^+$, and $(N^-_t)_{t\geq 0}$ is a Poisson process with
  rate $\lambda^-$; $Z^+_k$ $\left(Z_k^-\right)$,  $k=1, 2,...$, are independent and
  identically distributed random variables; moreover, $(W_t)_{t\geq 0}$, $(N^+_t)_{t\geq 0}$,
   $(N^-_t)_{t\geq 0}$, $\{Z^+_k; k=1,2,\ldots\}$ and $\{Z^-_k; k=1,2,\ldots\}$ are
   independent mutually; finally, the density functions of $Z_1^+$ and  $Z_1^-$ are given
   respectively by
\begin{equation}
p^+(z)=\sum_{k=1}^{m^+}\sum_{j=1}^{m_k}c_{kj}\frac{(\eta_k)^jz^{j-1}}{(j-1)!}
e^{-\eta_k z}, \ \ z > 0,
\end{equation}
and
\begin{equation}
p^-(z)=\sum_{k=1}^{n^-}\sum_{j=1}^{n_k}d_{kj}\frac{(\vartheta_k)^jz^{j-1}}{(j-1)!}
e^{-\vartheta_k z}, \ \ z > 0.
\end{equation}
Besides, it is assumed that $\eta_i \neq \eta_j$ and $\vartheta_i \neq \vartheta_j$
for $i \neq j$.

The following Lemma 4.1 is a combination of Lemma 1.1 in [14]
and Proposition 1 (v) in [8]. It characterizes the roots of $\psi(z)=q$  with
\begin{equation}
\begin{split}
\psi(z):
&= \ln\left(\mathbb E\left[e^{iz X_1}\right]\right)=
\lambda^+\left(\sum_{k=1}^{m^+}\sum_{j=1}^{m_k}c_{kj}\left(\frac{\eta_k}{\eta_k-iz}\right)^j-1\right)\\
&+  iz\mu -\frac{\sigma^2}{2}z^2+\lambda^-\left(\sum_{k=1}^{n^-}\sum_{j=1}^{n_k}d_{kj}
\left(\frac{\vartheta_k}{\vartheta_k+iz}\right)^j-1\right),\ \ z \in \mathbb R,
\end{split}
\end{equation}
where  $\sigma > 0$.

\begin{Lemma}
(i) For almost all $q > 0$, the equation $\psi(z)= q$ has, in the set $Im(z) < 0$, a total
of $M=\sum_{k=1}^{m^+}m_k+1$ distinct simple solutions $-i\beta_{1,q}$, $-i\beta_{2,q}$,
 $\ldots$, $-i\beta_{M,q}$, ordered such that
\begin{equation}
0 < \beta_{1,q}< Re(\beta_{2,q})\leq \cdots \leq Re(\beta_{M,q}).
\end{equation}

(ii) For almost all $q > 0$, the equation $\psi(z) = q$ has, in the set $Im(z) > 0$, a total
of $N= \sum_{k=1}^{n^-}n_k+1$ distinct simple roots $i\gamma_{1,q}$, $i\gamma_{2,q}$, $\ldots$,
 $i\gamma_{N,q}$, ordered such that
\begin{equation}
0< \gamma_{1,q}< Re(\gamma_{2,q})\leq \cdots \leq Re(\gamma_{N,q}).
\end{equation}

(iii) There are only finite numbers of $q > 0$ such that $\psi(z)= q$ has a root with
multiplicity larger than one.
\end{Lemma}

From now on, we denote by $\mathbb Q$ the set of $q > 0$ such that the equation $\psi(z) = q$
 only has simple roots.

Our objection in this section is proving that Theorem 3.1 holds for the process $X$ in (4.1),
and this will be done in Subsections 4.1 and 4.2.

\subsection{The case of $q, p+q \in \mathbb Q$}
\label{sec:4.1}
In this subsection, we want to show that Theorem 3.1 holds for $X$ given by (4.1) and
$q, p+q \in \mathbb Q$. First, for $y > b$ and $q, p+q \in \mathbb Q$, the expression for
$V_q(x)=\mathbb E_x\left[e^{-p\int_{0}^{e(q)}\textbf{1}_{\{X_s \leq b\}}ds}
\textbf{1}_{\{X_{e(q)} > y\}}\right]$ with $X$ in (4.1) is summarized in Proposition 4.1,
 whose proof is left to the Appendix.

\begin{Proposition}
For $X$ in (4.1), $y>b$, and $q, p+q \in \mathbb Q$,
\begin{equation}
\begin{split}
&V_q(x)=
\left\{\begin{array}{cc}
\sum_{k=1}^{M}U_{k}e^{\beta_{k,p+q}(x-b)},& x < b, \\
\sum_{k=1}^{M}H_{k}e^{\beta_{k,q}(x-y)}+\sum_{k=1}^{N}P_{k}
e^{\gamma_{k,q}(b-x)}, & b < x < y,\\
1+\sum_{k=1}^{N}Q_{k}e^{\gamma_{k,q}(y-x)}+\sum_{k=1}^{N}P_{k}
e^{\gamma_{k,q}(b-x)}, & x > y,
\end{array}\right.
\end{split}
\end{equation}
where $H_k$ and $Q_k$ are given by (A.18) and (A.19); $U_k$ and $P_k$ are given by rational expansion:
\begin{equation}
\begin{split}
&\sum_{i=1}^{M}\frac{U_i}{x-\beta_{i,p+q}}-
\sum_{i=1}^{N}\frac{P_i}{x+\gamma_{i,q}}-\sum_{i=1}^{M}\frac{H_i}{x-\beta_{i,q}}e^{\beta_{i,q}(b-y)}\\
&=\frac{\prod_{k=1}^{m^+}(x-\eta_k)^{m_k}\prod_{k=1}^{n^-}(x+\vartheta_k)^{n_k}}
{\prod_{i=1}^{M}(x-\beta_{i,p+q})\prod_{i=1}^{N}(x+\gamma_{i,q})}\times\\
&\sum_{k=1}^{M}
\frac{\prod_{i=1}^{M}(\beta_{k,q}-\beta_{i,p+q})\prod_{i=1}^{N}(\beta_{k,q}+\gamma_{i,q})}
{\prod_{i=1}^{m^+}(\beta_{k,q}-\eta_i)^{m_i}\prod_{i=1}^{n^-}(\beta_{k,q}+\vartheta_i)^{n_i}}
\frac{-H_k}{x-\beta_{k,q}}e^{\beta_{k,q}(b-y)}.
\end{split}
\end{equation}
\end{Proposition}

\begin{Remark}
The expressions for $U_k$ and $P_k$ can be easily obtained from (4.8), but here, we are not
interested in these expressions. Hence, the corresponding results are omitted for the sake
of brevity.
\end{Remark}

\begin{Remark}
We comment that $V_q(x)$ in (4.7) is continuous on $(-\infty,\infty)$. In fact, from (4.7),
 it is enough to show
\begin{equation}
V_q(b-)=V_q(b+) \ \ and \ \ V_q(y-)=V_q(y+).
\end{equation}
These two identities can be derived from (A.22) and the following result:
\begin{equation}
\sum_{i=1}^{M}H_i-\sum_{i=1}^{N}Q_i-1=0,
\end{equation}
which can be obtained from (A.24) by letting $\theta \uparrow \infty$.
\end{Remark}

\begin{Lemma}
For $X$ in (4.1), $y>b$, and $q, p+q \in \mathbb Q$, we have
\begin{equation}
\begin{split}
&\int_{-\infty}^{\infty}e^{-\phi(x-b)}\left(V_q(x)-\mathbb P_x\left(X_{e(q)}> y\right)\right)dx\\
&=\mathbb E\left[e^{\phi \underline{X}_{e(q)}}\right]\mathbb E\left[e^{\phi
\overline{X}_{e(p+q)}}\right]
\int_{0}^{\infty}F_0(x+y-b)e^{\phi x}dx, \ \ Re(\phi)=0,
\end{split}
\end{equation}
where
\begin{equation}
F_0(x)=\sum_{i=1}^{M}e^{-\beta_{i,q}x}\prod_{k=1}^{M}\frac{\beta_{i,q}-\beta_{k,p+q}}{\beta_{k,p+q}}
\prod_{k=1,k \neq i}^{M}\frac{\beta_{k,q}}{\beta_{i,q}-\beta_{k,q}}, \ \ x \geq 0.
\end{equation}
In addition,
\begin{equation}
\begin{split}
\int_{0}^{\infty} e^{-s x}F_0(x)dx
&=\sum_{i=1}^{M}\prod_{k=1}^{M}\frac{\beta_{i,q}-\beta_{k,p+q}}{\beta_{k,p+q}}
\prod_{k=1,k \neq i}^{M}\frac{\beta_{k,q}}{\beta_{i,q}-\beta_{k,q}}\frac{1}{\beta_{i,q}+s}\\
&=\frac{1}{s}\left(\frac{\mathbb E\left[e^{-s \overline{X}_{e(q)}}\right]}
{\mathbb E\left[e^{-s \overline{X}_{e(p+q)}}\right]}-1\right), \ \ s>0.
\end{split}
\end{equation}
\end{Lemma}

\begin{proof}
From (4.7), for $Re(\phi) =0$, some direct calculations yield
\begin{equation}
\begin{split}
&\int_{-\infty}^{\infty}e^{-\phi(x-b)}dV_q(x)
=\sum_{i=1}^{M}\frac{U_i\phi}{\beta_{i,p+q}-\phi}
+\sum_{i=1}^{N}\frac{P_i\phi}{\phi+\gamma_{i,q}}
\\
&-\sum_{i=1}^{M}\frac{H_i\phi}{\beta_{i,q}-\phi}e^{\beta_{i,q}(b-y)}+e^{\phi(b-y)}
\left(\sum_{i=1}^{M}\frac{H_i \beta_{i,q}}{\beta_{i,q}-\phi}-\sum_{i=1}^{N}
\frac{Q_i \gamma_{i,q}}{\gamma_{i,q}+\phi}\right)\\
&=\phi\mathbb E\left[e^{\phi \underline{X}_{e(q)}}\right]\mathbb E\left[e^{\phi
\overline{X}_{e(p+q)}}\right]
\int_{0}^{\infty}F_0(x+y-b)e^{\phi x}dx\\
&+e^{\phi(b-y)}
\left(\sum_{i=1}^{M}\frac{H_i \beta_{i,q}}{\beta_{i,q}-\phi}-\sum_{i=1}^{N}
\frac{Q_i \gamma_{i,q}}{\gamma_{i,q}+\phi}\right)\\
&=\phi\mathbb E\left[e^{\phi \underline{X}_{e(q)}}\right]\mathbb E\left[e^{\phi
\overline{X}_{e(p+q)}}\right]
\int_{0}^{\infty}F_0(x+y-b)e^{\phi x}dx+e^{\phi (b-y)}\psi_q^+(-\phi)\psi_q^-(\phi),
\end{split}
\end{equation}
where $\psi_q^+(\cdot)$ and $\psi_q^-(\cdot)$ are given respectively by (A.1) and (A.4);
in the first equality, we have used the identity
$\sum_{i=1}^{M}\left(U_i-H_ie^{\beta_{i,q}(b-y)}\right)-\sum_{i=1}^{N}P_i=0$ (see (A.22));
the second equality follows from (4.8), (A.1), (A.4) and (A.18) with $F_0(x)$ given by (4.12);
the third one is due to (4.10) and (A.24).

For fixed $y$, we have
\begin{equation}
\begin{split}
&\int_{-\infty}^{\infty}e^{-\phi(x-y)}d \mathbb P_x\left(X_{e(q)}> y\right)
=\mathbb E\left[e^{\phi X_{e(q)}}\right]
\\
&=\mathbb E\left[e^{\phi \overline{X}_{e(q)}}\right]\mathbb E\left[e^{\phi \underline{X}_{e(q)}}\right]
=\psi_q^+(-\phi)\psi_q^-(\phi),
\end{split}
\end{equation}
where the second equality is due to the well-known Wiener-Hopf factorization (see, e.g.,
Theorem 6.16 in [10]). Then, from (4.15), applying integration by parts will lead to
\begin{equation}
\begin{split}
&\int_{-\infty}^{\infty}e^{-\phi(x-b)}dV_q(x)
-e^{\phi (b-y)}\psi_q^+(-\phi)\psi_q^-(\phi)\\
&
=\phi \int_{-\infty}^{\infty}e^{-\phi(x-b)}\left(V_q(x)-\mathbb P_x\left(X_{e(q)}> y\right)\right)dx.
\end{split}
\end{equation}

From (4.14) and (4.16),  we arrive at (4.11). The second equality in (4.13) is a direct result
 of (A.1) and (4.17) in the following Lemma 4.3.
 \qed
\end{proof}

\begin{Lemma}
For given constants $\tilde{x}_1, \ldots, \tilde{x}_{n}$, which satisfy
$\tilde{x}_i \neq \tilde{x}_j$ for $i \neq j$, it holds that
\begin{equation}
\sum_{i=1}^{n}\frac{\prod_{k=1}^{m}(\tilde{x}_i-\hat{x}_k)}{\prod_{k=1,k\neq i}^{n}
(\tilde{x}_i-\tilde{x}_k)}=0,
\end{equation}
where $m < n-1$ and  $\hat{x}_1, \ldots, \hat{x}_{m}$ are arbitrary constants.
\end{Lemma}

\begin{proof}
Note first that
\begin{equation}
\frac{x \prod_{i=1}^{m}(x-\hat{x}_i)}{\prod_{i=1}^{n}(x-\tilde{x}_i)}=
\sum_{i=1}^{n}\frac{\prod_{k=1}^{m}(\tilde{x}_i-\hat{x}_k)}{\prod_{k=1,k\neq i}^{n}
(\tilde{x}_i-\tilde{x}_k)}\frac{x}{x-\tilde{x}_i}.
\end{equation}
Since $m < n-1$,  formula (4.17) is derived from (4.18) by letting $x \uparrow \infty$.
\qed
\end{proof}

\begin{Proposition}
For $X$ in (4.1), $q>0$ and $p>-q$ such that $q,p+q \in \mathbb Q$, Theorem 3.1 holds, i.e.,
\begin{equation}
V_q(x)-\mathbb P_x\left(X_{e(q)}> y\right)=J_0(b-x;y-b), \ \ y\geq b,
\end{equation}
with
\begin{equation}
J_0(x;y-b)=\int_{-\infty}^{x}F_0(x-z+y-b)dK_{q}(z), \ \ x \in \mathbb R,
\end{equation}
where $K_{q}(x)$ is given by (3.4) with $X$ in (4.1); $F_0(x)$ is given by (4.12) and its
Laplace transform is given by (4.13).
\end{Proposition}

\begin{proof}
Since $V_q(x)$  is a continuous function of $x$ (see Remark 4.2) and
$\mathbb P_x\left(X_{e(q)} > y\right)$, as a function of $x$, is also continuous with respect
to $x$ (as $\mathbb P\left(X_{e(q)} = z\right)=0$ for all $z \in \mathbb R$, see Lemma 2.1),
 the integrand on the left-hand side of (4.11) is continuous with respect to $x$.

As the distributions of $\overline{X}_{e(p+q)}$ and $\underline{X}_{e(q)}$ have continuous
density functions (see Lemma A.1), $K_q(x)$ (see (3.4)) also has a density function. This
result and the continuity of $F_0(x)$ on $(0,\infty)$ will lead to that $J_0(x;y-b)$ in
(4.20) is continuous  on $(-\infty, \infty)$. Finally, the definition of $J_0(x;y-b)$ gives
\begin{equation}
\mathbb E\left[e^{\phi \underline{X}_{e(q)}}\right]\mathbb E\left[e^{\phi \overline{X}_{e(p+q)}}\right]
\int_{0}^{\infty}F_0(x+y-b)e^{\phi x}dx=\int_{-\infty}^{\infty}J_0(x;y-b)e^{\phi x}dx.
\end{equation}
From (4.11) and (4.21), we can derive the conclusion that (4.19) holds for $y > b $.

Next, it is obvious that
\begin{equation}
\lim_{y \downarrow b}\mathbb P_x\left(X_{e(q)}> y\right)=\mathbb P_x\left(X_{e(q)}> b\right).
\end{equation}
Besides, we have $\lim_{y \downarrow b}J_0(x;y-b)=J_0(x;0)$,
which is due to the fact that $F_0(x)$ is bounded and  continuous on $(0,\infty)$ (see (4.12))
and the dominated convergence theorem. Since (A.14) holds, applying the dominated convergence
theorem again gives
\[
\lim_{y \downarrow b}\mathbb E_x\left[e^{-p\int_{0}^{e(q)}\textbf{1}_{\{X_s \leq b\}}ds}
\textbf{1}_{\{X_{e(q)} > y\}}\right]=\mathbb E_x\left[e^{-p\int_{0}^{e(q)}
\textbf{1}_{\{X_s \leq b\}}ds}\textbf{1}_{\{X_{e(q)} > b\}}\right].
\]
Therefore, we conclude that (4.19) also holds for $y=b$.
\qed
\end{proof}

\subsection{The case of $q \in \mathbb Q^c$ or $ p+q \in \mathbb Q^c$}
\label{sec:4.2}
In this subsection, for given $q>0$ and $p>-q$, we assume that either $q\in \mathbb Q^c$
or $p+q \in \mathbb Q^c$ holds.

From Lemma 4.1 (iii), we can find a sequence of $q_n$ such that $q_n, p+q_n \in \mathbb Q$
and $\lim_{n \uparrow \infty}q_n \downarrow q$. For each $n$, the result in Proposition 4.2
leads to

\begin{equation}
V_{q_n}(x)-\mathbb P_x\left(X_{e(q_n)} > y\right)=J^n_0(b-x;y-b),\ \  y>b,
\end{equation}
with
\begin{equation}
J^n_0(x;y-b)=\int_{-\infty}^{x}F_0^n(x-z+y-b)dK_{q_n}(z),  \ \ x \in \mathbb R,
\end{equation}
where $K_{q_n}(x)$ is the convolution of $\underline{X}_{e(q_n)}$ and $\overline{X}_{e(p+q_n)}$
under $\mathbb P$, and
\begin{equation}
\begin{split}
&\int_{0}^{\infty} e^{-s x}F_0^n(x)dx
=\frac{1}{s}\left(\frac{\mathbb E\left[e^{-s \overline{X}_{e(q_n)}}\right]}
{\mathbb E\left[e^{-s \overline{X}_{e(p+q_n)}}\right]}-1\right),\ \  s > 0.
\end{split}
\end{equation}

\begin{Lemma}
It holds that
\begin{equation}
\begin{split}
&\lim_{n \uparrow \infty}V_{q_n}(x)=V_q(x),\ \
\lim_{n \uparrow \infty}\mathbb P_x\left(X_{e(q_n)} > y\right)=\mathbb P_x\left(X_{e(q)} > y\right).
\end{split}
\end{equation}
And for $K_{q}(x)$ given by (3.4) with $X$ in (4.1), we have
\begin{equation}
\lim_{n \uparrow \infty} K_{q_n}(x) =K_q(x), \ \ x  \in  \mathbb R.
\end{equation}
\end{Lemma}

\begin{proof}
It follows from the definition of (3.1)  that
\begin{equation}
\begin{split}
&V_{q_n}(x)=\int_{0}^{\infty}
q_n e^{-q_n t}\mathbb E_x\left[e^{-p\int_{0}^{t}\textbf{1}_{\{X_s \leq b\}}ds}
\textbf{1}_{\{X_{t} > y\}}\right]dt.
\end{split}
\end{equation}
Since $q_n > q$ and $p+q>0$, formula (4.26) is obtained from  the application of the dominated
 convergence theorem. Similarly, for $Re(\phi)\geq 0$, we have
\begin{equation}
\begin{split}
\lim_{n \uparrow \infty}\mathbb E\left[e^{\phi \underline{X}_{e(q_n)}}\right]
=\mathbb E\left[e^{\phi \underline{X}_{e(q)}}\right], \  \lim_{n \uparrow \infty}
\mathbb E\left[e^{-\phi \overline{X}_{e(p+q_n)}}\right]
=\mathbb E\left[e^{-\phi \overline{X}_{e(p+q)}}\right].
\end{split}
\end{equation}
Formula (4.29) means that $\overline{X}_{e(p+q_n)}$ and $\underline{X}^n_{e(q_n)}$ converge
respectively to $\overline{X}_{e(p+q)}$ and $\underline{X}_{e(q)}$ in distribution. Thus,
recalling the definition of $K_q(x)$ in (3.4) and Proposition 3.1, we derive (4.27).
\qed
\end{proof}

For the function $F_0^n(x)$ in (4.25) and $F_1(x)$ in (3.5) with $X$ given by (4.1), we have
the following lemma.

\begin{Lemma}
(i) $F_0^n(x)$ is uniformly convergence to $F_1(x)$ on $[0,\infty]$.

(ii) $F_1(x)$, $F_0^1(x), F_0^2(x)$, $\ldots$, are uniformly bounded.
\end{Lemma}

\begin{proof}
First, formula (3.7) leads to
\begin{equation}
F_0^n(\infty)=0=F_1(\infty),
\end{equation}
and $
F_0^n(0) =e^{-\int_{0}^{\infty}\frac{1}{t}e^{-q_nt}\left(1-e^{-pt}\right)
\mathbb P\left(X_t>0\right)dt}-1.
$
Thus
\begin{equation}
\lim_{n\uparrow \infty}F_0^n(0) = F_1(0)=
  e^{-\int_{0}^{\infty}\frac{1}{t}e^{-qt}\left(1-e^{-pt}\right)\mathbb P\left(X_t>0\right)dt}-1.
\end{equation}
In addition, similar to (4.29), we also have
\begin{equation}
 \lim_{n \uparrow \infty}\mathbb E\left[e^{-s \overline{X}_{e(q_n)}}\right]
 =\mathbb E\left[e^{-s \overline{X}_{e(q)}}\right], \ \ s > 0.
\end{equation}

(1) Assume $p > 0$. It follows from (3.13), (4.29) and (4.32) that
\begin{equation}
\lim_{n \uparrow \infty}\int_{0^-}^{\infty} e^{-s x}d(F_0^n(x)+1)
=\int_{0^-}^{\infty} e^{-s x}d(F_1(x)+1).
\end{equation}
Since $F_1(x)$ is continuous on $(0,\infty)$, applying the continuity theorem for Laplace
 transforms (see, e.g., Theorem 2a on page 433 in Feller [6]) to (4.33) gives
\begin{equation}
\lim_{n \uparrow \infty} F_0^n(x)+1=F_1(x)+1, \ \ x > 0.
\end{equation}
Since $F_0^n(x)$ is increasing, continuous and bounded on $(0,\infty)$ (see Proposition 3.2 (ii)).
From (4.30), (4.31) and (4.34),  we  deduce the result that $F_0^n(x)$ is uniformly convergence
to $F_1(x)$ on $[0,\infty]$.
Besides, formula (3.8) produces
\begin{equation}
-1\leq F_1(x), \ \ F_0^{n}(x) \leq 0, \ \ for \ \ n=1,2,\ldots,
\end{equation}

(2) Assume $-q < p < 0$. For each $n=1,2,\ldots$, we obtain from (3.6) that
\begin{equation}
F_0^n(x)=e^{-\Pi_2^n(0,\infty)}\left(G_{21}^n(x)-G_{22}^n(x)\right)-1,
\end{equation}
where
\begin{equation}
\Pi_2^n(0,\infty):=\int_{0}^{\infty}\frac{1}{t}e^{-q_nt}(e^{-pt}-1)\mathbb P\left(X_t > 0\right)dt,
\end{equation}
and $G_{21}^n(x)$ and $G_{22}^n(x)$ are given respectively by (see (2.5), (2.6) and (2.8))
\begin{equation}
2\int_{0^-}^{\infty}e^{-s x}dG^n_{21}(x)=e^{-\Pi^n_2(0,\infty)}
\frac{\mathbb E\left[e^{-s\overline{X}_{e(q_n)}}\right]}
{\mathbb E\left[e^{-s\overline{X}_{e(p+q_n)}}\right]}
+e^{\Pi^n_2(0,\infty)}\frac{\mathbb E\left[e^{-s\overline{X}_{e(p+q_n)}}\right]}
{\mathbb E\left[e^{-s\overline{X}_{e(q_n)}}\right]},
\end{equation}
and
\begin{equation}
2\int_{0^-}^{\infty}e^{-s x}dG^n_{21}(x)=e^{\Pi^n_2(0,\infty)}
\frac{\mathbb E\left[e^{-s\overline{X}_{e(p+q_n)}}\right]}
{\mathbb E\left[e^{-s\overline{X}_{e(q_n)}}\right]}-e^{-\Pi^n_2(0,\infty)}
\frac{\mathbb E\left[e^{-s\overline{X}_{e(q_n)}}\right]}
{\mathbb E\left[e^{-s\overline{X}_{e(p+q_n)}}\right]}.
\end{equation}

Obviously, we have
\begin{equation}
\lim_{n \uparrow \infty}\Pi_2^n(0,\infty)=\Pi_2(0,\infty):
=\int_{0}^{\infty}\frac{1}{t}e^{-qt}(e^{-pt}-1)\mathbb P\left(X_t > 0\right)dt.
\end{equation}
Then from (4.29) and (4.32), we immediately derive that
\begin{equation}
\int_{0^-}^{\infty}e^{-s x}dG^n_{2i}(x)=\int_{0^-}^{\infty}e^{-s x}dG_{2i}(x), \ \ i=1,2,
\end{equation}
where $G_{21}(x)$ and $G_{22}(x)$ are given by (2.5), (2.6) and (2.7) with $X$ in (4.1).
Applying the continuity theorem for Laplace transforms again leads to
\begin{equation}
\lim_{n \uparrow \infty} G_{2i}^n(x)=G_{2i}(x), \ \ for \ \ x > 0 \ \ and \ \ i=1,2.
\end{equation}
Recall (2.18) and (2.19). From (4.40), for $i=1,2$, it is easy to show that
\begin{equation}
\lim_{n \uparrow \infty} G_{2i}^n(0)=G_{2i}(0)\ \ and \ \
\lim_{n \uparrow \infty} G_{2i}^n(\infty)=G_{2i}(\infty) < \infty.
\end{equation}
Since $G_{21}(x)$ and $G_{22}(x)$ are measures, it follows from (4.42) and (4.43)
that $G_{2i}^n(x)$ is uniformly convergence to $G_{2i}(x)$ on $[0,\infty]$, where $i=1,2$.

Therefore, from (3.6) and (4.36), we arrive at the result that $F_0^n(x)$ is uniformly
convergence to $F_1(x)$ on $[0,\infty]$. Since $ q_n > q$, formula (3.9) gives us
\begin{equation}
|F_1(x)|,|F_0^n(x)| < e^{2\int_{0}^{\infty}\frac{1}{t}e^{-qt}(e^{-pt}-1)dt}+1,
\end{equation}
this completes the proof.
\qed
\end{proof}

\begin{Proposition}
For $X$ in (4.1), Theorem 3.1 is valid for the case of $q\in \mathbb Q^c$ or $p+q \in \mathbb Q^c$.
\end{Proposition}

\begin{proof}
From Lemma 4.5 (i), we know $F_0^n(x)$ is uniformly convergence to $F_1(x)$ on $[0, \infty)$.
Thus if $x_n \rightarrow x \geq 0$, then
\begin{equation}
\lim_{n \uparrow \infty}F_0^n(x_n)=F_1(x).
\end{equation}
For $y \geq b$, $J_0^n(x;y-b)$ can be written as (see (4.24))
\begin{equation}
J_0^n(x;y-b)=\mathbb E\left[F_0^n(x-Z_0^n+y-b)\textbf{1}_{\{Z_0^n < x\}}\right],
\end{equation}
where the law of $Z_0^n$ is given by  $K_{q_n}(z)$. Besides, as Lemma 4.5 (ii) and
Proposition 3.1 hold, we deduce from (4.27), (4.45) and (4.46) that
\begin{equation}
\begin{split}
\lim_{n \uparrow \infty} J_0^n(x;y-b)
=J_1(x;y-b)=\mathbb E\left[F_1(x-Z_0+y-b)\textbf{1}_{\{Z_0 <x\}}\right],
\end{split}
\end{equation}
where the distribution of $Z_0$ is given by $K_q(z)$; and we have used the bounded convergence
theorem  in the derivation of (4.47).

Therefore, the desired result is deduced by letting $n\uparrow \infty$ in (4.23) and using
(4.26) and (4.47).
\qed
\end{proof}

\section{Proof of Theorem 3.1}
\label{sec:5.}
In this section, the details on the derivation of Theorem 3.1 are given. The following
technical lemma is important, and one can refer to Proposition 1 in Asmussen et al. [2]
for its proof.

\begin{Lemma}
For any  given L\'evy process $X=(X_t)_{t \geq 0}$, there exists a sequence of
$X^n=(X^n_t)_{t \geq 0}$ with the form of (4.1) such that
\begin{equation}
\lim_{n \uparrow \infty} \sup_{s \in [0,t]}|X^n_s-X_s|=0, \ \ almost \ \ surely.
\end{equation}
\end{Lemma}

\begin{Remark}
Note that $X^n$ in (5.1) is assumed to have a Gaussian component. A particular case is that the process $X$ in (5.1) is a pure jump process and at first sight Proposition 1 in Asmussen et al. [2] does not cover this special case. Here, for any given L\'evy process $X$, we remind the reader that
\[
\lim_{n \uparrow \infty} \sup_{s \in [0,t]}|\left(X_s+\frac{1}{n}W_s\right)-X_s|=\lim_{n \uparrow \infty} \frac{1}{n}\overline{W}_T=0,
\]
where $W_t$ is a Brownian motion. This means that Proposition 1 in Asmussen et al. [2] also holds for the above  case.
\end{Remark}

For each $X^n$ and $y \geq b$, Propositions 4.2 and 4.3 give
\begin{equation}
V_q^n(x)-\mathbb P_x\left(X^n_{e(q)}> y\right)=J^n_1(b-x;y-b),
\end{equation}
with
\begin{equation}
V_q^n(x):=\mathbb E_x\left[e^{-p\int_{0}^{e(q)}\rm{\bf{1}}_{\{X^n_s \leq b\}}ds}
\rm{\bf{1}}_{\{X^n_{e(q)}>y\}}\right],
\end{equation}
and
\begin{equation}
J^n_1(x;y-b)=\int_{-\infty}^{x}F_1^n(x-z+y-b)dK_q^n(z),
\end{equation}
where $K_q^n(x)$ is the convolution of $\underline{X}^n_{e(q)}$ and $\overline{X}^n_{e(p+q)}$
under $\mathbb P$, and
\begin{equation}
\begin{split}
&\int_{0}^{\infty} e^{-s x}F_1^n(x)dx
=\frac{1}{s}\left(\frac{\mathbb E\left[e^{-s \overline{X}^n_{e(q)}}\right]}
{\mathbb E\left[e^{-s \overline{X}^n_{e(p+q)}}\right]}-1\right),\ \  s > 0.
\end{split}
\end{equation}

\begin{Lemma}
It holds that
\begin{equation}
\begin{split}
\lim_{n \uparrow \infty}V_q^n(x)=V_q(x),\ \
\lim_{n \uparrow \infty}\mathbb P_x\left(X^n_{e(q)} > y\right)=\mathbb P_x\left(X_{e(q)} > y\right),
\end{split}
\end{equation}
and
\begin{equation}
\lim_{n \uparrow \infty} K^n_{q}(x) =K_q(x), \ \ x  \in  \mathbb R,
\end{equation}
where $K_{q}(x)$ is given by (3.4).
\end{Lemma}

\begin{proof}
Since Lemmas 2.1 and 5.1 hold, the dominated convergence theorem will lead to
\begin{equation}
\begin{split}
&\lim_{n \uparrow \infty}\mathbb P_x\left(X^n_{e(q)} > z \right)=\lim_{n \uparrow \infty}q
\mathbb E\left[\int_{0}^{\infty}e^{-qt}\textbf{1}_{\{X^n_t>z\}}dt\right]
=\mathbb P_x\left(X_{e(q)} > z \right).
\end{split}
\end{equation}
Similarly, we have (note that $p+q>0$)
\begin{equation}
\begin{split}
&\lim_{n \uparrow \infty}V_q^n(x):=\lim_{n \uparrow \infty}q
\mathbb E_x\left[\int_{0}^{\infty}e^{-qt} e^{-p\int_{0}^{t}\rm{\bf{1}}_{\{X^n_s \leq b\}}ds}
\rm{\bf{1}}_{\{X^n_{t}>y\}}dt\right]=V_q(x).
\end{split}
\end{equation}

In addition, it is known that (see, e.g., Lemma 13.4.1 of Whitt [17])
\begin{equation}
|\overline{X}^n_t-\overline{X}_t|\leq \sup_{0 \leq s \leq t}|X^n_s-X_s|\
and \ |\underline{X}^n_t-\underline{X}_t|\leq \sup_{0 \leq s \leq t}|X^n_s-X_s|,
\end{equation}
which combined with (5.1), yields
\begin{equation}
\lim_{n \uparrow \infty}\mathbb E\left[e^{-s \overline{X}^n_{e(q)}}\right]
=q\lim_{n \uparrow \infty}\int_{0}^{\infty}e^{-qt}\mathbb E\left[e^{-s\overline{X}^n_{t}}\right]dt
=\mathbb E\left[e^{-s \overline{X}_{e(q)}}\right],
\end{equation}
and
\begin{equation}
 \lim_{n \uparrow \infty}\mathbb E\left[e^{s \underline{X}^n_{e(q)}}\right]
 =q\lim_{n \uparrow \infty}\int_{0}^{\infty}e^{-qt}\mathbb E\left[e^{s\underline{X}^n_{t}}\right]dt
 =\mathbb E\left[e^{s \underline{X}_{e(q)}}\right],
\end{equation}
where $s,q >0$.
Formulas (5.11) and (5.12) mean that $\overline{X}^n_{e(p+q)}$ and $\underline{X}^n_{e(q)}$
converge respectively to $\overline{X}_{e(p+q)}$ and $\underline{X}_{e(q)}$ in distribution,
thus formula (5.7) is derived by recalling Proposition 3.1.
\qed
\end{proof}

\begin{Lemma}
For the continuous functions $F_1(x)$ in (3.5) and $F_1^n(x)$ in (5.5), the following results hold.

(i) $F_1^n(x)$ is uniformly convergence to $F_1(x)$ on $[0,\infty)$.

(ii) $F_1(x)$, $F_1^1(x), F_1^2(x)$, $\ldots$, are uniformly bounded.
\end{Lemma}

\begin{proof}
Note that (5.11) holds for any $q>0$. Besides, from (5.1),  we have
\[
\lim_{n\uparrow \infty}\int_{0}^{\infty}\frac{1}{t}e^{-qt}\left(1-e^{-pt}\right)
\mathbb P\left(X^n_t>0\right)dt=\int_{0}^{\infty}\frac{1}{t}e^{-qt}\left(1-e^{-pt}\right)
\mathbb P\left(X_t>0\right)dt,
\]
which is due to the result that $
\mathbb P\left(X_{t}=0\right)=0$ for Lebesgue almost every $t>0$ (see Lemma 2.1) and the
dominated convergence theorem (note that $p+q>0$).

The remaining proof of this lemma is similar to that of Lemma 4.5, thus the details are
omitted for simplicity.
\qed
\end{proof}

\textit{Proof of Theorem 3.1}
First, Lemma 5.3 (i) states  that if $x_n \rightarrow x  \geq 0$, then
\begin{equation}
\lim_{n \uparrow \infty}F_1^n(x_n)=F_1(x).
\end{equation}
Next, for $y  \geq b$, $J_1^n(x;y-b)$ in (5.4) can be written as
\begin{equation}
J_1^n(x;y-b)=\mathbb E\left[F_1^n(x-Z_1^n+y-b)\textbf{1}_{\{Z_1^n < x\}}\right],
\end{equation}
where the law of $Z_1^n$ is given by  $K_q^n(z)$. Since $F_1(x)$, $F_1^1(x)$, $F_1^2(x)$, $\ldots$,
 are uniformly bounded (see Lemma 5.3 (ii)). Applying the bounded convergence theorem to (5.14)
 and using (5.7), (5.13) and Proposition 3.1, we obtain
\begin{equation}
\begin{split}
\lim_{n \uparrow \infty} J_1^n(x;y-b)
=J_1(x;y-b)=\mathbb E\left[F_1(x-Z_1+y-b)\textbf{1}_{\{Z_1<x\}}\right],
\end{split}
\end{equation}
where the distribution of $Z_1$ is given by $K_q(z)$. Therefore, letting $n \uparrow \infty$ in
(5.2), we derive Theorem 3.1 from  (5.6) and (5.15).\qed

\section{Examples}
In Corollary 3.3 (see also Remark 3.10), we obtain expresses for the following expectation:
\begin{equation}
\mathbb E_x\left[e^{-p\int_{0}^{e(q)}\rm{\bf{1}}_{\{X_s \leq b\}}ds}
\rm{\bf{1}}_{\{X_{e(q)} \in dy\}}\right]=q\int_{0}^{\infty}e^{-qt} \mathbb E_x\left[e^{-p\int_{0}^{t}\rm{\bf{1}}_{\{X_s \leq b\}}ds}
\rm{\bf{1}}_{\{X_{t} \in dy\}}\right]dt,
\end{equation}
where $p,q > 0$ and $X$ is a general L\'evy process but not a compound Poisson process.
For some L\'evy process $X$, this quantity $\mathbb E_x\left[e^{-p\int_{0}^{e(q)}\rm{\bf{1}}_{\{X_s \leq b\}}ds}
\rm{\bf{1}}_{\{X_{e(q)} \in dy\}}\right]$ has more explicit expressions. And in the following,
we will give some examples.

\renewcommand{\theExample}{6.\arabic{Example}}

\begin{Example}
Let $X$ be a hyper-exponential jump diffusion process, i.e., the process $X$ is given by (4.1) but the jump distributions $p^+(z)$ and $p^-(z)$ are simplified as:
\[
p^+(z)=\sum_{k=1}^{m^+}c_k\eta_ke^{-\eta_k z},  \ \
p^-(z)=\sum_{k=1}^{n^-}d_{k}\vartheta_ke^{-\vartheta_k z}, \ \ z > 0,
\]
where $c_k,\eta_k,\vartheta_k, d_k>0$ and $\sum_{k=1}^{m^+}c_k=1=\sum_{k=1}^{n^-}d_k$.

In this case, the equation $\psi(z)=q$ for any $q>0$ only has real and simple roots, where $\psi(z)=\ln\left(\mathbb E\left[e^{z X_1}\right]\right)$ (see Lemma 2.1 in [4]). The distributions of $\overline{X}_{e(q)}$ and $\underline{X}_{e(q)}$ for any $q>0$ have semi-explicit expressions, whose forms are the same as (A.2) and (A.5). In addition, the function $F_1(x)$ defined by (3.5) has the same form as $F_0(x)$ given by (4.12), and $\hat{F}_2(x)$ in (3.28) has a similar expression. Thus the left-hand side of (3.31) can be written in a more explicit form.

In fact, after some simple calculations, we obtain that formula (3.31) will reduce to the results given by Corollary 3.9 in [18].

\qed
\end{Example}

\begin{Example}
Assume that $X$ is a spectrally negative L\'evy process. First, we give some results on such a L\'evy process $X$ and refer the reader to chapter $8$ of [10] for more details.

It is known that for $\lambda >0$,
\[
\psi(\lambda)=\ln (\mathbb E\left[e^{\lambda X_1}\right])
=\frac{1}{2}\sigma^2\lambda^2+ \gamma \lambda+\int_{-\infty}^{0}(e^{\lambda x}-1-\lambda x\textbf{1}_{\{x>-1\}})\Pi(dx),
\]
where $\gamma \in \mathbb R$ and $\sigma \geq 0$; the L\'evy measure $\Pi$ has a support of $(-\infty,0)$ such that $\int_{-\infty}^{0}(x^2\wedge 1)\Pi(dx)< \infty$.
Besides, if $\sigma=0$ and $\int_{-1}^{0}|x|\Pi(dx)<\infty$, then $X$ has bounded variation and
$
\psi(\lambda)= d \lambda +\int_{-\infty}^{0}(e^{\lambda x}-1)\Pi(dx)
$
with  $d=\gamma - \int_{-1}^{0} x\Pi(dx)>0$.

Define
\begin{equation}
\Phi(q)=\sup\{\lambda \geq 0: \psi(\lambda)=q\},\ \ \ for \ \ q>0.
\end{equation}
For given $q>0$, the q-scale function $W^{(q)}(x)$ is strictly increasing and continuous on $(0,\infty)$ and its Laplace transform satisfies
\begin{equation}
\int_{0}^{\infty}e^{-s x}W^{(q)}(x)dx=\frac{1}{\psi(s)-q}, \ \ for \ \ s >\Phi(q).
\end{equation}
In addition, $W^{(q)}(x)=0$ for $x < 0$ and $W^{(q)}(0):=\lim_{x \downarrow 0}W^{(q)}(x)$.

Next, we derive formulas for $F_1(x)$, $\hat{F}_2(x)$ and $K_q(x)$ given respectively by (3.5), (3.28) and (3.4).

From (8.2) in [10], we know
\begin{equation}
\mathbb E\left[e^{-s\overline{X}_{e(q)}}\right]=\frac{\Phi(q)}{\Phi(q)+s} \ \ and \ \ \mathbb E\left[e^{s\underline{X}_{e(q)}}\right]=\frac{q}{\Phi(q)}\frac{\Phi(q)-s}{q-\psi(s)},  \ \ s,q > 0.
\end{equation}

(i) It follows directly from (3.13) and (6.4) that $F_1(0)=\frac{\Phi(q)}{\Phi(p+q)}-1$ and
\begin{equation}
F_1(dx)=\frac{\Phi(p+q)-\Phi(q)}{\Phi(p+q)}\Phi(q)e^{-\Phi(q)x}dx, \ \ x >0.
\end{equation}

(ii) Note that (see (3.28))
\[
\int_{0}^{\infty}e^{-s x} \hat{F}_2(dx)+\hat{F}_2(0)+1=\mathbb E\left[e^{s\underline{X}_{e(p+q)}}\right]/
\mathbb E\left[e^{s\underline{X}_{e(q)}}\right].
\]
So formula (6.4) gives $\hat{F}_2(0)=\frac{(p+q)\Phi(q)}{q\Phi(p+q)}-1$. Then, for $s>\max\{\Phi(p+q),\Phi(q)\}$, some straightforward calculations will yield
\[
\begin{split}
&\int_{0}^{\infty}e^{-s x} \hat{F}_2(dx)=\mathbb E\left[e^{s\underline{X}_{e(p+q)}}\right]/
\mathbb E\left[e^{s\underline{X}_{e(q)}}\right]-\frac{(p+q)\Phi(q)}{q\Phi(p+q)}\\
&=\frac{(p+q)\Phi(q)(\Phi(q)-\Phi(p+q))}{q\Phi(p+q)(s-\Phi(q))}+
\frac{(p+q)p\Phi(q)}{q\Phi(p+q)}\frac{1+\frac{\Phi(q)-\Phi(p+q)}
{s-\Phi(q)}}{\psi(s)-(p+q)}.
\end{split}
\]
The above result and formula (6.3) will lead to
\begin{equation}
\begin{split}
\frac{\hat{F}_2(dx)}{dx}
=\frac{(p+q)\Phi(q)(\Phi(q)-\Phi(p+q))}{q\Phi(p+q)}e^{\Phi(q)x}+
\frac{(p+q)p\Phi(q)}{q\Phi(p+q)}\hat{f}_2(x), \ \  x>0,
\end{split}
\end{equation}
where
\begin{equation}
\hat{f}_2(x)=W^{(p+q)}(x)+(\Phi(q)-\Phi(p+q))\int_{0}^{x}e^{\Phi(q)(x-z)}W^{(p+q)}(z)dz, \ \ x \in \mathbb R.
\end{equation}

(iii) Since $\mathbb P\left(\overline{X}_{e(p+q)} \in dx\right)=\Phi(p+q)e^{-\Phi(p+q)x}dx$ for $x>0$ (see (6.4)) and
(see, e.g., formula (8.20) on page 219 of [10])
\[
\mathbb P\left(-\underline{X}_{e(q)} \in dx\right)=\frac{q}{\Phi(q)}W^{(q)}(dx)-qW^{(q)}(x)dx, \ \ x \geq 0.
\]
From (3.4), we can rewrite $K_q(dx)$ as
\[
\Phi(p+q)\left\{\int_{-\infty}^{0^+}e^{-\Phi(p+q)(x-z)}
\mathbb P\left(\underline{X}_{e(q)} \in dz\right)
-\int_{x}^{0^+}e^{-\Phi(p+q)(x-z)}
\mathbb P\left(\underline{X}_{e(q)} \in dz\right)\right\}.
\]
Integration by parts shows that
\[
\int_{x}^{0^+}e^{-\Phi(p+q)(x-z)}W^{(q)}(-dz)=W^{(q)}(-x)+\Phi(p+q)\int_{x}^{0}e^{-\Phi(p+q)(x-z)}W^{(q)}(-z)dz.
\]
From (6.4) and the last three formulas, we will derive the following result after some simple computations.
\begin{equation}
\begin{split}
\frac{K_q(dx)}{dx}
=\frac{q}{p}\Phi(p+q)\left(\frac{\Phi(p+q)}{\Phi(q)}-1\right)e^{-\Phi(p+q)x}-\frac{q\Phi(p+q)}{\Phi(q)}k_q(x),
\ \ x \in \mathbb R,
\end{split}
\end{equation}
where
\begin{equation}
k_q(x)=W^{(q)}(-x)+\left(\Phi(p+q)-\Phi(q)\right)e^{-\Phi(p+q)x}\int_{x}^{0}e^{\Phi(p+q)z}
W^{(q)}(-z)dz.
\end{equation}

In addition,
the definition of $\hat{L}_q(x)$ (see Corollary 3.2 (ii)) gives
\[\hat{L}_q(dx)=K_q(-dx), \ \  for \ \ x \in \mathbb R,\]
which can be obtained from (6.8).

Finally, as $\hat{F}_2(0)=\frac{(p+q)\Phi(q)}{q\Phi(p+q)}-1$ and $F_1(0)=\frac{\Phi(q)}{\Phi(p+q)}-1$,
we can rewrite formula (3.33) as follows:
\begin{equation}
\begin{split}
&\mathbb E_x\left[e^{-p\int_{0}^{e(q)}\rm{\bf{1}}_{\{X_s \leq b\}}ds}
\rm{\bf{1}}_{\{X_{e(q)} \in dy\}}\right]\\
&=\left\{\begin{array}{cc}
\frac{\Phi(q)}{\Phi(p+q)}K_q(dy-x)+\int_{b-x}^{y-x}F_1(dy-x-z)dK_{q}(z),& y \geq b,\\
\frac{q}{p+q}\left(\frac{(p+q)\Phi(q)}{q\Phi(p+q)}\hat{L}_q(x-dy)+\int_{x-b}^{x-y}
\hat{F}_2(x-dy-z)d \hat{L}_{q}(z)\right),& y \leq  b,
\end{array}\right.
\end{split}
\end{equation}
which combined with (6.5), (6.6) and (6.8), leads to (see Appendix B for the details on the derivation)
\begin{equation}
\begin{split}
&\mathbb E_x\left[e^{-p \int_{0}^{e(q)}\rm{\bf{1}}_{\{X_s<b\}}ds}\rm{\bf{1}}_{\{
X_{e(q)}\in dy\}}\right]=-qW_{x-b}^{(q,p)}(x-y)dy\\
&+\frac{q}{p}(\Phi(p+q)-\Phi(q))H^{(p+q,-p)}(x-b)H^{(q,p)}(b-y)dy,
\end{split}
\end{equation}
where
\begin{equation}
\begin{split}
&H^{(p+q,-p)}(x-b)=e^{\Phi(p+q)(x-b)}\left[1-p\int_{0}^{x-b}e^{-\Phi(p+q)z}W^{(q)}(z)dz\right],\\
&H^{(q,p)}(b-y)=e^{\Phi(q)(b-y)}\left[1+p\int_{0}^{b-y}e^{-\Phi(q)z}W^{(p+q)}(z)dz\right],
\end{split}
\end{equation}
and
\begin{equation}
\begin{split}
W_{x-b}^{(q,p)}(x-y)
=W^{(q)}(x-y)+p\int_{x-b}^{x-y}W^{(p+q)}(x-y-z)W^{(q)}(z)dz.
\end{split}
\end{equation}

Therefore, formula (6.11) recovers the result obtained in previous research, see (19) in [11] or (12) in [19].
\end{Example}

\section{Conclusion}
\label{sec:6}
In this paper, we investigate the occupation times of a general L\'evy process. Formulas for
the Laplace transform of the joint distribution of an arbitrary L\'evy process (which is not
a compound Poisson process) and its occupation times are derived. The approach used is novel
and the result has some applications in finance. Particularly, the application of our result
to price occupation time derivatives is a potential direction of our future research.

\begin{appendix}

\section{The proof of Proposition 4.1}

In this section, we derive Proposition 4.1 and present some preliminary results before
starting the derivation.

The following Lemma A.1 gives the distributions of $\overline{X}_{e(q)}$ and
$\underline{X}_{e(q)}$ for any $q \in \mathbb Q$, where Lemma A.1 (i) is taken from
Theorem 2.2 and Corollary 2.1 in  Lewis and Mordecki [14]; and Lemma A.1 (ii) is a
straightforward application of Lemma A.1 (i) to the dual process $-X$.

\begin{Lemma}
(i) For any $q \in \mathbb Q$ and $Re(s) \geq 0$,
\begin{equation}
\begin{split}
\mathbb E\left[e^{-s \overline{X}_{e(q)}}\right]
&=\prod_{k=1}^{m^+}\left(\frac{s+\eta_k}{\eta_k}\right)^{m_k}
\prod_{k=1}^{M}\left(\frac{\beta_{k,q}}{s+\beta_{k,q}}\right)
=\sum_{k=1}^{M}\frac{C^q_k}{s+\beta_{k,q}}:=\psi_q^+(s),
\end{split}\tag{A.1}
\end{equation}
and for $z \geq 0$,
\begin{equation}
\mathbb P\left(\overline{X}_{e(q)}\in dz\right)=\sum_{k=1}^{M}C^q_{k}e^{-\beta_{k,q}z}dz,\tag{A.2}
\end{equation}
where
\begin{equation}
\frac{C^q_i}{\beta_{i,q}}=\prod_{k=1}^{m^+}\left(\frac{\eta_k-\beta_{i,q}}{\eta_k}\right)^{m_k}
\prod_{k=1,k \neq i}^{M}\frac{\beta_{k,q}}{\beta_{k,q}-\beta_{i,q}},\ \ for \ \  1\leq i \leq M.\tag{A.3}
\end{equation}

(ii) For any $q \in \mathbb Q$ and $Re(s) \geq 0$,
\begin{equation}
\begin{split}
\mathbb E\left[e^{s\underline{X}_{e(q)}}\right]
&=\prod_{k=1}^{n^-}\left(\frac{s+\vartheta_k}{\vartheta_k}\right)^{n_k}
\prod_{k=1}^{N}\left(\frac{\gamma_{k,q}}{s+\gamma_{k,q}}\right)
=\sum_{k=1}^{N}\frac{D^q_{k}}{s+\gamma_{k,q}}:=\psi_q^-(s),
\end{split}\tag{A.4}
\end{equation}
and for $z \leq 0$,
\begin{equation}
\mathbb P\left(\underline{X}_{e(q)}\in dz\right)=\sum_{k=1}^{N}
D^q_{k}e^{\gamma_{k,q}z}dz,\tag{A.5}
\end{equation}
where
\begin{equation} \frac{D^q_j}{\gamma_{j,q}}=\prod_{k=1}^{n^-}\left(\frac{\vartheta_k-\gamma_{j,q}}
{\vartheta_k}\right)^{n_k}
\prod_{k=1,k \neq j}^{N}\left(\frac{\gamma_{k,q}}{\gamma_{k,q}-\gamma_{j,q}}\right),
\ \ for \ 1\leq j \leq N.\tag{A.6}
\end{equation}
\end{Lemma}

Although (A.1) and (A.4) hold for $Re(s) \geq 0$ only,  $\psi_q^+(s)$ in (A.1) and
$\psi_q^-(s)$ in (A.4) are treated  as two rational functions of $s$ in what follows.
In addition, for any $a \in \mathbb R$, define
\begin{equation}
\tau_{a}^{+} := \inf\{t\geq 0: X_t > a\}\ \ and \ \
\tau_{a}^{-}:= \inf\{t\geq 0: X_t < a \}.\tag{A.7}
\end{equation}

Lemma A.2 summarizes the results on the one-sided exit problems of $X$, and its proof
is very easy by applying Lemma A.1 and the following two results (see, Corollary 2 and
formula (4) in Alili and Kyprianou [1])
\[
\mathbb E\left[e^{-q \tau_{x}^-+s(X_{\tau_{x}^-}-x)}\right]=\frac{\mathbb E\left[
e^{s(\underline{X}_{e(q)}-x)}\textbf{1}_{\{\underline{X}_{e(q)}< x \}}\right]}{\mathbb E\left[
e^{s\underline{X}_{e(q)}}\right]}, \ \ s \geq 0 \ \  and  \ \ x \leq 0,
\]
and
\[
\mathbb E\left[e^{-q \tau_{x}^+-s (X_{\tau_{x}^+}-x)}\right]=\frac{\mathbb E\left[
e^{-s(\overline{X}_{e(q)}-x)}\textbf{1}_{\{\overline{X}_{e(q)} > x \}}\right]}{\mathbb E\left[
e^{-s\overline{X}_{e(q)}}\right]}, \ \ x, s \geq 0.
\]

\begin{Lemma}
(1) For $q \in  \mathbb Q$ and $x, y \leq 0$, we have

\begin{equation}
\begin{split}
\mathbb E\left[e^{-q\tau_{x}^-}\textbf{1}_{\{X_{\tau_{x}^-}-x \in dy\}}\right]=D^q_0(x)\delta_0(dy)+
\sum_{k=1}^{n^-}\sum_{j=1}^{n_k}D^q_{kj}(x)\frac{(\vartheta_k)^j(-y)^{j-1}}{(j-1)!}e^{\vartheta_k y}dy,
\end{split}\tag{A.8}
\end{equation}
with $D^q_0(x)$ and $D^q_{kj}(x)$ given by rational expansion:
\begin{equation}
\begin{split}
&D^q_0(x)+\sum_{k=1}^{n^-}\sum_{j=1}^{n_k}D^q_{kj}(x)\left(\frac{\vartheta_k}{\vartheta_k+ s}\right)^j
=
\frac{1}{\psi_q^-(s)}\sum_{k=1}^{N}D^q_{k}\frac{e^{\gamma_{k,q} x}}{s+\gamma_{k,q}},  \  x\leq 0,
\end{split}\tag{A.9}
\end{equation}
where $\psi_q^-(s)$ is a rational function and is given by (A.4).

(2) For $q \in \mathbb Q$ and $x, y \geq  0$,
\begin{equation}
\begin{split}
&\mathbb E\left[e^{-q \tau_{x}^+}\textbf{1}_{\{X_{\tau_{x}^+}-x \in dy\}}\right]=
C^q_0(x)\delta_0(dy)+
\sum_{k=1}^{m^+}\sum_{j=1}^{m_k}C^q_{kj}(x)\frac{(\eta_k)^jy^{j-1}}{(j-1)!}e^{-\eta_k y}dy,
\end{split}\tag{A.10}
\end{equation}
with $C^q_0(x)$ and $C^q_{kj}(x)$ given by rational expansion:
\begin{equation}
\begin{split}
&C^q_0(x)+\sum_{k=1}^{m^+}\sum_{j=1}^{m_k}C^q_{kj}(x)\left(\frac{\eta_k}{\eta_k+ s}\right)^j
=\frac{1}{\psi_q^+(s)}\sum_{k=1}^{M}C^q_{k}\frac{e^{-\beta_{k,q} x}}{s+\beta_{k,q}}, \ \ x\geq 0,
\end{split}\tag{A.11}
\end{equation}
where $\psi_q^+(s)$ is a rational function and is given by (A.1).
\end{Lemma}

\begin{Remark}
From (A.9), we see that $D^q_0(x)$ and $D^q_{kj}(x)$, for $1\leq k\leq n^-$ and
$1\leq j \leq n_k$, are linear combinations of $e^{\gamma_{i,q} x}$ for $1\leq i \leq N$.
Similarly, formula (A.11) leads to that $C^q_0(x)$ and $C^q_{kj}(x)$, for $1\leq k\leq m^+$
and $1\leq j \leq m_k$, are linear combinations of $e^{\beta_{i,q} x}$ for  $1 \leq i \leq M$.
\end{Remark}

\begin{Lemma}
For any $\theta >0$ and $s \neq -\eta_1, \ldots, -\eta_{m^+}$ with $\theta \neq s$,
\begin{equation}
\begin{split}
&\int_{0}^{\infty}e^{-\theta x}C^q_0(x)dx+\sum_{k=1}^{m^+}\sum_{j=1}^{m_k} \int_{0}^{\infty}
e^{-\theta x}C^q_{kj}(x)dx\left(\frac{\eta_k}{\eta_k+s}\right)^j=\frac{1}{s-\theta}
\left(\frac{\psi_q^+(\theta)}
{\psi_q^+(s)}-1\right),
\end{split}\tag{A.12}
\end{equation}
and for any $\theta >0$ and $s \neq -\vartheta_1, \ldots, -\vartheta_{n^-}$ with $\theta \neq s$,
\begin{equation}
\begin{split}
&\int_{-\infty}^{0}e^{\theta x}D^q_0(x)dx +\sum_{k=1}^{n^-}\sum_{j=1}^{n_k} \int_{-\infty}^{0}
e^{\theta x}D^q_{kj}(x)dx\left(\frac{\vartheta_k}{\vartheta_k+s}\right)^j=\frac{1}{s-\theta}
\left(\frac{\psi_q^-(\theta)}
{\psi_q^-(s)}-1\right).
\end{split}\tag{A.13}
\end{equation}
\end{Lemma}

\begin{proof}
These results can be obtained from (A.9) and (A.11) after some direct algebraic manipulations.
 Here, we only remind that
\[
\int_{0}^{\infty}e^{-\theta x}\frac{ e^{-\beta_{k,q} x}}{s+\beta_{k,q}}dx=\frac{1}{s-\theta}
\left(\frac{1}{\theta+\beta_{k,q}}-\frac{1}{s+\beta_{k,q}}\right).
\]
\qed
\end{proof}

\textbf{Proof of Proposition 4.1} The derivation consists of three steps.

Step 1. For given $y > b$, considering  the  function defined in (3.1), we have
\begin{equation}
\begin{split}
V_q(x)
&= \mathbb E_x\left[e^{-p\int_{0}^{e(q)}\textbf{1}_{\{X_s \leq b\}}ds}
\textbf{1}_{\{X_{e(q)} > y\}}\right]\\
&\leq \left\{\begin{array}{cc}
1,& if \ \ p\geq 0,\\
 \mathbb E_x\left[e^{-pe(q)}\right]=\frac{q}{p+q}, & if -q < p < 0.
\end{array}\right.
\end{split}\tag{A.14}
\end{equation}

For $x<b$,   we can obtain from the strong Markov property of the process $X$ and the lack
of memory property of $e(q)$ that
\begin{equation}
\begin{split}
V_q(x)
&=\mathbb E_x\left[e^{-p\int_{0}^{e(q)}\textbf{1}_{\{X_s \leq b\}}ds}
\textbf{1}_{\{X_{e(q)} > y\}}\textbf{1}_{\{e(q) > \tau_b^+\}}\right]\\
&=\mathbb E_x\left[e^{-p\tau_b^+-p\int_{\tau_b^+}^{e(q)}\textbf{1}_{\{X_s \leq b\}}ds}
\textbf{1}_{\{X_{e(q)} > y\}}\textbf{1}_{\{e(q) > \tau_b^+\}}\right]\\
&=\mathbb E_x\left[e^{-(p+q)\tau_b^+}V_q(X_{\tau^+_b})\right]\\
&=\sum_{k=1}^{m^+}\sum_{j=1}^{m_k}C^{\xi}_{kj}(b-x)\int_{0}^{\infty}
\frac{(\eta_k)^jz^{j-1}}{(j-1)!}e^{-\eta_k z}V_q(b+z)dz\\
&+ C^{\xi}_0(b-x)V_q(b)=\sum_{k=1}^{M}U_{k}e^{\beta_{k,\xi}(x-b)},  \ \ x < b,
\end{split}\tag{A.15}
\end{equation}
where $\xi = p+q$ and $U_1, \ldots, U_M$ are proper constants and do not depend on $x$;
the fourth equality follows from (A.10) and the final one is due to Remark A.1.

Similarly, for $x > b$,  we can derive
\begin{equation}
\begin{split}
V_q(x)
&=\mathbb E_x\left[e^{-p\int_{\tau_b^-}^{e(q)}\textbf{1}_{\{X_s \leq b\}}ds}
\textbf{1}_{\{X_{e(q)} > y\}}\textbf{1}_{\{e(q)> \tau_b^-\}}
\right]\\
&+\mathbb E_x\left[\textbf{1}_{\{X_{e(q)} > y\}}\textbf{1}_{\{e(q) \leq \tau_b^-\}}
\right]\\
&=\mathbb E_x\left[e^{-q \tau_{b}^{-}}V_q(X_{\tau_{b}^{-}})\right]+\mathbb P_x(X_{e(q)} > y,
 \underline{X}_{e(q)} \geq b)\\
&=\sum_{k=1}^{n^-}\sum_{j=1}^{n_k}D^q_{kj}(b-x)\int_{-\infty}^{0}V_q(b+z)
\frac{(\vartheta_k)^j(-z)^{j-1}}{(j-1)!}e^{\vartheta_kz}dz\\
&+D^q_0(b-x)V_q(b) +\mathbb P_x(X_{e(q)} > y, \underline{X}_{e(q)} \geq b),
\end{split}\tag{A.16}
\end{equation}
where the third equality follows from (A.8).

The well-known Wiener-Hopf factorization (see, e.g., Theorem 6.16 in [10]) gives
that $X_{e(q)}-\underline{X}_{e(q)}$ is independent of $\underline{X}_{e(q)}$ and
is equal in distribution to $\overline{X}_{e(q)}$ under $\mathbb P$. This result and formulas
 (A.2) and (A.5) yield
\begin{equation}
\begin{split}
\mathbb P_x(X_{e(q)}
& > y, \underline{X}_{e(q)} \geq b)
=
\int_{b-x}^{0}
\mathbb P(X_{e(q)}-\underline{X}_{e(q)} > y - x -z, \underline{X}_{e(q)}\in dz) \\
&=
\int_{b-x}^{0}\mathbb P(\overline{X}_{e(q)} > y - x -z)\mathbb P(\underline{X}_{e(q)}\in dz)\\
&=
\left\{\begin{array}{cc}
\sum_{k=1}^{M}H_{k}e^{\beta_{k,q}(x-y)}+\sum_{k=1}^{N}\hat{P}_{k}
e^{\gamma_{k,q}(b-x)}, & b <  x \leq y,\\
1+\sum_{k=1}^{N}Q_{k}e^{\gamma_{k,q}(y-x)}+\sum_{k=1}^{N}\hat{P}_{k}
e^{\gamma_{k,q}(b-x)}, & x \geq y,
\end{array}\right.
\end{split}\tag{A.17}
\end{equation}
where we remind the reader that $\mathbb P\left(\overline{X}_{e(q)}> z\right)=1$ for $z\leq 0$;
for $ k=1,2,\ldots, M,$
\begin{equation}
\begin{split}
H_k=\frac{C^q_k}{\beta_{k,q}}\sum_{j=1}^{N}\frac{D^q_j}{\beta_{k,q}+\gamma_{j,q}},
\end{split}\tag{A.18}
\end{equation}
and for $k=1,2,\ldots, N$,
\begin{equation}
\begin{split}
Q_k=D^q_k\sum_{i=1}^{M}\frac{C^q_i}{\beta_{i,q}(\beta_{i,q}+\gamma_{k,q})}-\frac{D^q_k}
{\gamma_{k,q}}, \
\hat{P}_k=-\sum_{i=1}^{M}\frac{C^q_i}{\beta_{i,q}}
\frac{D^q_k e^{\beta_{i,q}(b-y)}}{\beta_{i,q}+\gamma_{k,q}}.
\end{split}\tag{A.19}
\end{equation}

From (A.16), (A.17) and Remark A.1, we arrive at
\begin{equation}
V_q(x)=\left\{
\begin{array}{cc}
\sum_{k=1}^{M}H_{k}e^{\beta_{k,q}(x-y)}+\sum_{k=1}^{N}P_{k}
e^{\gamma_{k,q}(b-x)}, & b < x \leq y,\\
1+\sum_{k=1}^{N}Q_{k}e^{\gamma_{k,q}(y-x)}+\sum_{k=1}^{N}P_{k}
e^{\gamma_{k,q}(b-x)}, & x \geq y,
\end{array}\right.\tag{A.20}
\end{equation}
where $P_1, \ldots, P_N$ do not depend on $x$ and satisfy
\begin{equation}
\begin{split}
&\sum_{k=1}^{N}P_ke^{\gamma_{k,q}(b-x)}=-\sum_{i=1}^{M}\frac{C^q_i}{\beta_{i,q}}
\sum_{j=1}^{N}\frac{D^q_j e^{\gamma_{j,q}(b-x)}}{\beta_{i,q}+\gamma_{j,q}}e^{\beta_{i,q}(b-y)}
+D^q_0(b-x)V_q(b)\\
&+\sum_{k=1}^{n^-}\sum_{j=1}^{n_k}D^q_{kj}(b-x)\int_{-\infty}^{0}V_q(b+z)
\frac{(\vartheta_k)^j(-z)^{j-1}}{(j-1)!}e^{\vartheta_kz}dz, \ \  x > b.
\end{split}\tag{A.21}
\end{equation}

Formulas (A.15) and (A.20) imply that the remaining thing is to derive the expressions of
$U_k$ and $P_k$. In the second step, we establish the equations satisfied by
$U_1, \ldots, U_M$ and $P_1, \ldots, P_N$; and in the last step, we solve these equations.

Step 2. Since $\sigma > 0$, it is known that $\mathbb P\left(\tau_0^+=0\right)
=\mathbb P\left(\tau_0^-=0\right)=1$. Thus $V_q(x)$ is continuous at $b$ (see (A.15)
and (A.16)), i.e., $\lim_{x\uparrow b} V_q(x)=V_q(b)=\lim_{x\downarrow b}V_q(x)$, which
combined with (A.15) and (A.20), yields
\begin{equation}
\sum_{i=1}^{M}U_i=V_q(b)=\sum_{i=1}^{M}H_ie^{\beta_{i,q}(b-y)}+\sum_{i=1}^{N}P_i.\tag{A.22}
\end{equation}
Besides, we know that the derivative of $V_q(x)$ at $b$ is continuous, i.e.,
$V_q^{\prime}(b-)=V_q^{\prime}(b+)$\footnote{From the proof given in the Appendix A of
Wu and Zhou [18], we obtain that $V_q^{\prime}(b-)$ and $V_q^{\prime}(b+)$ must be
equal if they are existent. The existence of $V_q^{\prime}(b-)$ and $V_q^{\prime}(b+)$ can
 be seen from (A.15) and (A.20).}. Then, it follows from  (A.15) and (A.20) that
\begin{equation}
\sum_{i=1}^{M}U_i\beta_{i,\xi}=\sum_{i=1}^{M}H_i\beta_{i,q}e^{\beta_{i,q}(x-y)}-
\sum_{i=1}^{N}P_i\gamma_{i,q}.\tag{A.23}
\end{equation}

For all $\theta \in \mathbb C $ except at $\beta_{1,q}, \ldots, \beta_{M,q}$ and
$-\gamma_{1,q}, \ldots, -\gamma_{N,q}$, formulas (A.18), (A.19) and some straightforward
 computations lead to
\begin{equation}
\begin{split}
&\sum_{k=1}^{M}\frac{\theta H_k}{\beta_{k,q}-\theta}+1+\sum_{k=1}^{N}\frac{\theta Q_k}
{\theta+\gamma_{k,q}}=\sum_{k=1}^{M}\frac{\theta H_k}{\beta_{k,q}-\theta}+
\sum_{i=1}^{M}\sum_{j=1}^{N}\frac{C^q_iD^q_j}{\beta_{i,q}
\gamma_{j,q}}\\
&+\sum_{k=1}^{N}\frac{\theta}{\theta+\gamma_{k,q}}\left(
D^q_k\sum_{i=1}^{M}\frac{C^q_i}{\beta_{i,q}(\beta_{i,q}+\gamma_{k,q})}-\frac{D^q_k}
{\gamma_{k,q}} \sum_{i=1}^{M}\frac{C^q_i}{\beta_{i,q}}\right) \\
&=\sum_{i=1}^{M}\sum_{j=1}^{N}\left\{\frac{\theta C^q_iD^q_j}{\beta_{i,q}(\beta_{i,q}+\gamma_{j,q})
(\beta_{i,q}-\theta)}+\frac{C^q_iD^q_j}{\beta_{i,q}\gamma_{j,q}}-\frac{\theta C^q_iD^q_j}
{\gamma_{j,q}(\beta_{i,q}+\gamma_{j,q})
(\gamma_{j,q}+\theta)}\right\}\\
&=\sum_{i=1}^{M}\sum_{j=1}^{N}\frac{C^q_iD^q_j}{(\beta_{i,q}-\theta)
(\theta+\gamma_{j,q})}=\psi_q^+(-\theta)\psi_q^-(\theta),
\end{split}\tag{A.24}
\end{equation}
where we have used
$\sum_{i=1}^{M}\frac{C^q_i}{\beta_{i,q}}=1=\sum_{j=1}^{N}\frac{D^q_j}{\gamma_{j,q}}$
(let $s=0$ in (A.1) and (A.4)) in the first equality; the last equality follows from (A.1)
and (A.4).

For $1\leq k \leq m^+$ and $1\leq j \leq m_k$, we can write
\[
(-1)^{j-1}\int_{b_1}^{b_2}z^{j-1}e^{-\eta_k z} e^{\theta z} dz
=\frac{\partial^{j-1}}{\partial \eta^{j-1}}\left(
\int_{b_1}^{b_2}e^{-\eta z} e^{\theta z} dz\right)_{\eta=\eta_k},
\]
providing the integral $\int_{b_1}^{b_2}e^{-\eta z} e^{\theta z}dz$ exists. So we derive
via (A.20) that
\begin{equation}
\begin{split}
\int_{0}^{\infty}\frac{(\eta_k)^jz^{j-1}}{(j-1)!}e^{-\eta_k z}V_q(b+z)dz=
\sum_{i=1}^{N}\frac{P_i(\eta_k)^j}{(\eta_k+\gamma_{i,q})^j}+\sum_{i=1}^{M}\frac{
H_i(\eta_k)^j}{(\eta_k-\beta_{i,q})^j}e^{\beta_{i,q}(b-y)},
\end{split}\tag{A.25}
\end{equation}
where we have used the following result:
\begin{equation}
\frac{(\eta_k)^j(-1)^{j-1}}{(j-1)!}\frac{\partial^{j-1}}{\partial \eta^{j-1}}\left(
\frac{1}{\eta}e^{\eta(b-y)}\Big(\sum_{i=1}^{M}\frac{H_i\eta}{\beta_{i,q}-\eta}+
\sum_{i=1}^{N}\frac{\eta Q_i}{\eta+\gamma_{i,q}}+1\Big)\right)_{\eta=\eta_k}=0,\tag{A.26}
\end{equation}
which can be proved by using (A.24) and noting that
\begin{equation}
\frac{\partial^{j-1}}{\partial \eta^{j-1}}\left(\psi_q^+(-\eta)\right)_{\eta=\eta_k}=0, \ \
for \ \ 1\leq k \leq m^+ \ \ and \ \ 1\leq j\leq m_k.\tag{A.27}
\end{equation}

Then, it follows from (A.15), (A.22) and (A.25) that
\begin{equation}
\begin{split}
&\sum_{k=1}^{M}U_{k}e^{\beta_{k,\xi}(x-b)}= C^{\xi}_0(b-x)\left(\sum_{i=1}^{M}
H_ie^{\beta_{i,q}(b-y)}+\sum_{i=1}^{N}P_i\right)\\
&+\sum_{k=1}^{m^+}\sum_{j=1}^{m_k}C^{\xi}_{kj}(b-x)\left(
\sum_{i=1}^{N}\frac{P_i(\eta_k)^j}{(\eta_k+\gamma_{i,q})^j}+\sum_{i=1}^{M}\frac{
H_i(\eta_k)^j}{(\eta_k-\beta_{i,q})^j}e^{\beta_{i,q}(b-y)}\right).
\end{split}\tag{A.28}
\end{equation}
Multiplying both sides of (A.28) by $e^{\theta (x-b)}$ and taking an integration from
$-\infty$ to $b$ with respect to $x$, we obtain from (A.12) that
\begin{equation}
\begin{split}
\sum_{i=1}^{M}\frac{U_i}{\beta_{i,\xi}+\theta}
=&\sum_{i=1}^{M}\frac{H_ie^{\beta_{i,q}(b-y)}}{\theta+\beta_{i,q}}\left(1-\frac{\psi_{\xi}^+(\theta)}
{\psi_{\xi}^+(-\beta_{i,q})}\right)\\
&+\sum_{i=1}^{N}\frac{P_i}{\gamma_{i,q}-\theta}\left(\frac{\psi_{\xi}^+(\theta)}
{\psi_{\xi}^+(\gamma_{i,q})}-1\right).
\end{split}\tag{A.29}
\end{equation}
Since both sides of (A.29) are rational functions of $\theta$, it can be extended to the
whole plane except at $-\beta_{1,\xi}, \ldots, -\beta_{M,\xi}$. Note that
\[
\lim_{\theta \rightarrow -\beta_{i,q}} \frac{\psi_{\xi}^+(-\beta_{i,q})-\psi_{\xi}^+(\theta)}
{\theta+\beta_{i,q}}=-\psi_{\xi}^{+\prime}(-\beta_{i,q}),\
\lim_{\theta \rightarrow \gamma_{i,q}} \frac{\psi_{\xi}^+(\theta)-\psi_{\xi}^+(\gamma_{i,q})}
{\theta -\gamma_{i,q}}=\psi_{\xi}^{+\prime}(\gamma_{i,q}).
\]

Similarly, for $1\leq k \leq n^-$ and $1\leq j \leq n_k$, we can derive from (A.15) that
\begin{equation}
\int_{-\infty}^{0}V_q(b+z)\frac{(\vartheta_k)^j(-z)^{j-1}}{(j-1)!}e^{\vartheta_k z}dz
=\sum_{i=1}^{M}\frac{U_i(\vartheta_k)^j}{(\vartheta_k+\beta_{i,\xi})^j}.\tag{A.30}
\end{equation}
From (A.13), (A.21), (A.30) and the fact of $V_q(b)=\sum_{i=1}^{M}U_i$ (see (A.22)), it can
be shown that
\begin{equation}
\begin{split}
&\sum_{i=1}^{N}\frac{P_i}{\theta+\gamma_{i,q}}=\sum_{i=1}^{N}P_i\int_{b}^{\infty}e^{\theta(b-x)}
e^{\gamma_{i,q}(b-x)}dx\\
&=-\sum_{i=1}^{M}
\sum_{j=1}^{N}\frac{D^q_j}{\beta_{i,q}+\gamma_{j,q}}\frac{e^{\beta_{i,q}(b-y)}}{\theta+\gamma_{j,q}}
\frac{C^q_i}{\beta_{i,q}}+
\sum_{i=1}^{M}\frac{U_i}{\beta_{i,\xi}-\theta}\left(\frac{\psi_q^-(\theta)}
{\psi_q^-(\beta_{i,\xi})}-1\right).
\end{split}\tag{A.31}
\end{equation}
Furthermore, it holds that
\begin{equation}
\begin{split}
&-\sum_{i=1}^{M}
\sum_{j=1}^{N}\frac{D^q_j}{\beta_{i,q}+\gamma_{j,q}}\frac{e^{\beta_{i,q}(b-y)}}{\theta+\gamma_{j,q}}
\frac{C^q_i}{\beta_{i,q}}\\
&=-\sum_{i=1}^{M}\frac{C^q_i}{\beta_{i,q}}e^{\beta_{i,q}(b-y)}\frac{1}{\beta_{i,q}-\theta}
\sum_{j=1}^{N}D^q_j\left(\frac{1}{\theta+\gamma_{j,q}}-\frac{1}{\beta_{i,q}+\gamma_{j,q}}\right)\\
&=\sum_{i=1}^{M}\frac{H_i}{\beta_{i,q}-\theta}e^{\beta_{i,q}(b-y)}-\sum_{i=1}^{M}
\frac{C^q_i\psi_q^-(\theta)}{\beta_{i,q}(\beta_{i,q}-\theta)}e^{\beta_{i,q}(b-y)}.
\end{split}\tag{A.32}
\end{equation}
where the second equality follows from (A.18) and (A.4). Hence,
\begin{equation}
\begin{split}
&\sum_{i=1}^{N}\frac{P_i}{\theta+\gamma_{i,q}}=
\sum_{i=1}^{M}\frac{U_i}{\beta_{i,\xi}-\theta}\left(\frac{\psi_q^-(\theta)}
{\psi_q^-(\beta_{i,\xi})}-1\right)\\
&+\sum_{i=1}^{M}\frac{H_i}{\beta_{i,q}-\theta}e^{\beta_{i,q}(b-y)}-\sum_{i=1}^{M}
\frac{C^q_i\psi_q^-(\theta)}{\beta_{i,q}(\beta_{i,q}-\theta)}e^{\beta_{i,q}(b-y)},
\end{split}\tag{A.33}
\end{equation}
which holds for $\theta \in \mathbb C$ except at $-\gamma_{1,q}, \ldots, -\gamma_{N,q}$.

Therefore, for any given $1 \leq k \leq m^+$ and $0\leq j \leq m_k-1$, taking a derivative
on both sides of (A.29) with respect to $\theta $ up to $j$ order and letting $\theta$ equal
 to $-\eta_k$ will produce
\begin{equation}
\sum_{i=1}^{M}\frac{U_i(-1)^{j}}{(\beta_{i,\xi}-\eta_k)^{j+1}}+\sum_{i=1}^{M}\frac{P_i}{(\eta_k+
\gamma_{i,q})^{j+1}}
-\sum_{i=1}^{M}\frac{H_i(-1)^j}{(\beta_{i,q}-\eta_k)^{j+1}}e^{\beta_{i,q}(b-y)}=0,\tag{A.34}
\end{equation}
where the fact that $\frac{\partial ^{j}}{\partial \theta^{j}}
\Big(\psi_{\xi}^+(\theta)\Big)_{\theta=-\eta_k}=0$ (see (A.1)) is used in the derivation.
In a similar way, for any given $1 \leq k \leq n^-$ and $0\leq j \leq n_k-1$, applying
$\frac{\partial ^{j}}{\partial \theta^{j}}\Big(\psi_q^-(\theta)\Big)_{\theta=-\vartheta_k}=0$
(see (A.4)) to (A.33) yields
\begin{equation}
\sum_{i=1}^{M}\frac{U_i(-1)^{j}}{(\beta_{i,\xi}+\theta_k)^{j+1}}+\sum_{i=1}^{M}\frac{P_i}
{(\gamma_{i,q}-\theta_k)^{j+1}}
-\sum_{i=1}^{M}\frac{H_i(-1)^j}{(\beta_{i,q}+\theta_k)^{j+1}}e^{\beta_{i,q}(b-y)}=0.\tag{A.35}
\end{equation}

Step 3.
Consider the following rational function of $x$:
\begin{equation}
f(x)=\sum_{i=1}^{M}\frac{U_i}{x-\beta_{i,\xi}}-
\sum_{i=1}^{N}\frac{P_i}{x+\gamma_{i,q}}-\sum_{i=1}^{M}\frac{H_i}
{x-\beta_{i,q}}e^{\beta_{i,q}(b-y)}.\tag{A.36}
\end{equation}
For fixed $1 \leq k \leq m^+$ and $0\leq j \leq m_k-1$, (A.34) yields that
$\frac{\partial^j}{\partial x^j}\left(f(x)\right)_{x=\eta_k}=0$, which means that $\eta_k$ is
a root of the equation $f(x)=0$ with multiplicity $m_k$. Besides, from (A.35), for
$1 \leq k \leq n^-$, we obtain that $-\vartheta_k$ is a $n_k$-multiplicity root of $f(x)=0$.
These results give us (recall $M=\sum_{k=1}^{m^+}m_k+1$ and $N=\sum_{k=1}^{n^-} n_k+1$; see Lemma 4.1)
\begin{equation}
f(x)=\frac{\prod_{k=1}^{m^+}(x-\eta_k)^{m_k}\prod_{k}^{n^-}(x+\vartheta_k)^{n_k}
(l_0+l_1x+\cdots+l_{M+1}x^{M+1})}
{\prod_{i=1}^{M}(x-\beta_{i,\xi})\prod_{i=1}^{N}(x+\gamma_{i,q})\prod_{i=1}^{M}
(x-\beta_{i,q})},\tag{A.37}
\end{equation}
with some proper constants $l_0,l_1, \ldots, l_{M+1}$.

Formulas (A.36) and (A.37) produce
\[
l_{M+1}=\sum_{i=1}^{M}U_i-
\sum_{i=1}^{N}P_i-\sum_{i=1}^{M}H_ie^{\beta_{i,q}(b-y)},
\]
and
\[
l_{M}=\sum_{i=1}^{M}U_i\left(Sum+\beta_{i,\xi}\right)-
\sum_{i=1}^{N}P_i\left(Sum-\gamma_{i,q}\right)-\sum_{i=1}^{M}H_ie^{\beta_{i,q}(b-y)}
\left(Sum+\beta_{i,q}\right),
\]
where $Sum=-\sum_{i=1}^{M}\beta_{i,\xi}+
\sum_{i=1}^{N}\gamma_{i,q}-\sum_{i=1}^{M}\beta_{i,q}$. So formulas (A.22) and (A.23) lead
 to that $l_{M+1}=0$ and $l_{M}=0$. In addition, it is obvious that (see (A.36))
\begin{equation}
\lim_{x\rightarrow \beta_{i,q}}f(x)(x-\beta_{i,q})=-H_ie^{\beta_{i,q}(b-y)},
 \ \ 1 \leq i \leq M,\tag{A.38}
\end{equation}
thus
\begin{equation}
\begin{split}
&f(x)=\frac{\prod_{k=1}^{m^+}(x-\eta_k)^{m_k}\prod_{k=1}^{n^-}(x+\vartheta_k)^{n_k}}
{\prod_{i=1}^{M}(x-\beta_{i,\xi})\prod_{i=1}^{N}(x+\gamma_{i,q})}\times\\
&\sum_{k=1}^{M}
\frac{\prod_{i=1}^{M}(\beta_{k,q}-\beta_{i,\xi})\prod_{i=1}^{N}(\beta_{k,q}+\gamma_{i,q})}
{\prod_{i=1}^{m^+}(\beta_{k,q}-\eta_i)^{m_i}\prod_{i=1}^{n^-}(\beta_{k,q}+\vartheta_i)^{n_i}}
\frac{-H_k}{x-\beta_{k,q}}e^{\beta_{k,q}(b-y)}.
\end{split}\tag{A.39}
\end{equation}
Formulas (4.7) and (4.8) are derived from (A.15), (A.20), (A.36) and (A.39).
\qed

\section{The derivation of (6.11)}
In this section, we give the details on the derivation of (6.11), and we will divide the arguments into two cases.

\textbf{Case 1}: Assume $y\geq b$.

In this case, we have (see (6.12) and (6.13))
\[
W_{x-b}^{(q,p)}(x-y)=
W^{(q)}(x-y)\ \ and \ \
H^{(q,p)}(b-y)=e^{\Phi(q)(b-y)}.
\]
For any given $x<0$, exchanging the order of integration produces
\begin{equation}
\int_{x}^{0}e^{-\Phi(q)(x-z)}k_q(z)dz
=e^{-\Phi(p+q)x}\int_{x}^{0}e^{\Phi(p+q)z}W^{(q)}(-z)dz.
\end{equation}

Since $k_q(z)=W^{(q)}(-z)=0$ if $z>0$. So identity (B.1) holds for all $x \in \mathbb R$.
In addition, for any given $x \in \mathbb R$, we have
\[
\begin{split}
\int_{b-x}^{y-x}e^{-\Phi(q)(y-x-z)}k_q(z)dz
=
&e^{\Phi(q)(b-y)}\int_{b-x}^{0}e^{-\Phi(q)(b-x-z)}k_q(z)dz\\
&-
\int_{y-x}^{0}e^{-\Phi(q)(y-x-z)}k_q(z)dz.
\end{split}
\]

When $y\geq b$, applying (6.5), (6.8), (6.10) and the above three formulas, we can deduce (6.11) after some simple and straightforward computations.

\textbf{Case 2}: Assume $y\leq b$.

As $\hat{L}_q(dx)=K_q(-dx)$, it follows from (6.8) that
\[
\begin{split}
\frac{\hat{L}_q(dx)}{dx}=
\frac{q}{p}\Phi(p+q)\left(\frac{\Phi(p+q)}{\Phi(q)}-1\right)e^{\Phi(p+q)x}
-\frac{q\Phi(p+q)}{\Phi(q)}k_q(-x),
\end{split}
\]
where $k_q(-x)$ is given by (6.9).

Formula (B.1) can be rewritten as
\[
\int_{0}^{x}e^{\Phi(q)(x-z)}k_q(-z)dz=e^{\Phi(p+q)x}\int_{0}^{x}e^{-\Phi(p+q)z}W^{(q)}(z)dz, \ \
 for \ \ all \ \ x \in \mathbb R.
\]
Besides, for $x \in \mathbb R$, it is obvious that
\[
\begin{split}
\int_{x-b}^{x-y}e^{\Phi(q)(x-y-z)}k_q(-z)dz
&=\int_{0}^{x-y}e^{\Phi(q)(x-y-z)}k_q(-z)dz\\
&-
e^{\Phi(q)(b-y)}\int_{0}^{x-b}e^{\Phi(q)(x-b-z)}k_q(-z)dz.
\end{split}
\]
Similar to the derivation of (B.1), from (6.7), we can obtain
\[
\int_{0}^{x}e^{\Phi(p+q)(x-z)}\hat{f}_2(z)dz=e^{\Phi(q)x}\int_{0}^{x}e^{-\Phi(q)z}W^{(p+q)}(z)dz,
\]
which holds for all $x\in \mathbb R$ (note that $\hat{f}_2(z)=W^{(p+q)}(z)=0$ for $z<0$).

Finally, for all $x \in \mathbb R$, we will prove the following result in Lemma B.1.
\begin{equation}
\begin{split}
&\int_{x-b}^{x-y}\hat{f}_2(x-y-z)k_q(-z)dz=\int_{x-b}^{x-y}W^{(p+q)}(x-y-z)W^{(q)}(z)dz\\
&+(\Phi(p+q)-\Phi(q))\int_{0}^{x-b}e^{\Phi(p+q)(x-b-z)}W^{(q)}(z)dz
\int_{0}^{b-y}e^{\Phi(q)(b-y-z)}W^{(p+q)}(z)dz.
\end{split}
\end{equation}

For given $y\leq b$, applying the above results, (6.6) and (6.10) will derive (6.11). This derivation only involves some straightforward calculations, thus  we omit the details for brevity.

\begin{Lemma}
For given $y \leq b$, formula (B.2) holds for all $x \in \mathbb R$.
\end{Lemma}
\begin{proof}
For $s>\max\{\Phi(p+q),\Phi(q)\}$, we can derive via (6.3), (6.7) and (6.9) that
\[
\int_{0}^{\infty}e^{-s x} \int_{0}^{x}\hat{f}_2(x-z)k_q(-z)dzdx=\frac{1}{\psi(s)-q}\frac{1}{\psi(s)-(p+q)},
\]
which leads to (note that $k_q(-z)=W^{(q)}(z)=0$ if $z<0$)
\begin{equation}
\int_{0}^{x}\hat{f}_2(x-z)k_q(-z)dz=\int_{0}^{x}W^{(q)}(z)W^{(p+q)}(x-z)dz, \ \ for \ \ all\ \  x \in \mathbb R.
\end{equation}
So formula (B.2) holds for $x \leq b$ (note that the second item on the right-hand side of (B.2) equals to zero if $x \leq b$).


If $x>b$, we note first that
\begin{equation}
\begin{split}
\int_{x-b}^{x-y}dz \int_{0}^{x-y-z}dt_1\int_{-z}^{0}
\hat{w}(\cdot)dt_2
&= \int_{0}^{b-y}dt_1\int_{b-x}^{0}d t_2 \int_{x-b}^{x-y-t_1} \hat{w}(\cdot)dz\\
&+ \int_{0}^{b-y}dt_1\int_{t_1+y-x}^{b-x}d t_2 \int_{-t_2}^{x-y-t_1}\hat{w}(\cdot)dz,
\end{split}
\end{equation}
where
\[
\hat{w}(\cdot)=e^{\Phi(q)(x-y-z-t_1)}W^{(p+q)}(t_1)e^{\Phi(p+q)(z+t_2)}W^{(q)}(-t_2).
\]
Then, from (6.7), (6.9) and (B.4), we can show that (B.2) also holds for $x>b$ after some straightforward computations.
\qed
\end{proof}

\begin{Remark}
For any given $c\geq 0$ and $x>0$, the following identity holds
\begin{equation}
\begin{split}
&\int_{0}^{x}\hat{f}_2(x+c-z)k_q(-z)dz=\int_{0}^{x}W^{(p+q)}(x+c-z)W^{(q)}(z)dz\\
&-(\Phi(p+q)-\Phi(q))\int_{0}^{x}e^{\Phi(p+q)(x-z)}W^{(q)}(z)dz
\int_{0}^{c}e^{\Phi(q)(c-z)}W^{(p+q)}(z)dz,
\end{split}
\end{equation}
since both sides have the same Laplace transform.

In fact, for $s>\max\{\Phi(p+q),\Phi(q)\}$, exchanging the order of integration gives
\[
\begin{split}
&\int_{0}^{\infty}e^{-s x}\int_{0}^{x}\hat{f}_2(x+c-z)k_q(-z)dzdx
=
\int_{0}^{\infty}\int_{z}^{\infty}e^{-s x}\hat{f}_2(x+c-z)dxk_q(-z)dz\\
&=\int_{0}^{\infty}e^{-sz}k_q(-z)dz\left(e^{sc}\int_{0}^{\infty}e^{-sx}\hat{f}_2(x)dx-
\int_{0}^{c}e^{-s(x-c)}\hat{f}_2(x)dx\right),
\end{split}
\]
and
\[
\begin{split}
&\int_{0}^{\infty}e^{-s x}\int_{0}^{x}W^{(p+q)}(x+c-z)W^{(q)}(z)dzdx
\\
&=\int_{0}^{\infty}e^{-sz}W^{(q)}(z)dz\left(e^{sc}\int_{0}^{\infty}e^{-sx}W^{(p+q)}(x)dx-
\int_{0}^{c}e^{-s(x-c)}W^{(p+q)}(x)dx\right).
\end{split}
\]
Note that
\[
\int_{0}^{c}e^{-s(x-c)}\int_{0}^{x}e^{\Phi(q)(x-z)}W^{(p+q)}(z)dzdx=
\int_{0}^{c}\int_{z}^{c}e^{-s(x-c)}e^{\Phi(q)(x-z)}dxW^{(p+q)}(z)dz.
\]
Combining the above three formulas with (6.3), (6.7) and (6.9), we can deduce that both sides of (B.5) have the same Laplace transform.
\qed
\end{Remark}

\begin{Remark}
Note first that identity (B.5) holds also for $x \leq 0$ (since $k_q(-z)=W^{(q)}(z)=0$ if $z<0$).
For any given $x \in \mathbb R$ and $y \leq b$, we can write
\[
\begin{split}
\int_{x-b}^{x-y}\hat{f}_2(x-y-z)k_q(-z)dz
&=\int_{0}^{x-y}\hat{f}_2(x-y-z)k_q(-z)dz\\
&-
\int_{0}^{x-b}\hat{f}_2(x-b+(b-y)-z)k_q(-z)dz.
\end{split}
\]
Then formula (B.2) follows from (B.5).
\qed
\end{Remark}

\end{appendix}




\end{document}